\documentclass[10pt]{amsart}
\usepackage{amsmath,amssymb, amsthm, amsbsy}
\usepackage[dvips]{graphicx}
\usepackage[usenames,dvipsnames]{pstricks}
\usepackage{epsfig}
\usepackage{pst-grad} 
\usepackage{pst-plot} 

\newcommand{\R}{{\mathbb R}}
\newcommand{\N}{{\mathbb N}}
\newcommand{\Z}{{\mathbb Z}}

\newcommand{\SN}{{\mathbb S}^{N-1}}

\newcommand{\dyle}{\displaystyle}
\newcommand{\dint}{\dyle\int}

\newcommand{\e }{\varepsilon}

\newcommand{\dive }{\mathop{\rm div}}

\newcommand{\A}{{\mathbf A}}

\renewcommand{\ge }{\geqslant}
\renewcommand{\geq }{\geqslant}

\renewcommand{\leq }{\leqslant}

\newenvironment{pf}{\noindent{\sc Proof}.\enspace}{\hfill\qed\medskip}
\newenvironment{pfn}[1]{\noindent{\bf Proof of
    {#1}.\enspace}}{\hfill\qed\medskip}
\numberwithin{equation}{section}
\newtheorem{Theorem}{Theorem}[section]
\newtheorem{Corollary}[Theorem]{Corollary}
\newtheorem{Lemma}[Theorem]{Lemma}
\newtheorem{Proposition}[Theorem]{Proposition}
\theoremstyle{definition} 

\newtheorem{remark}[Theorem]{Remark}

\begin{document}
\title[Schr\"odinger with critical potentials]{Time decay of
scaling critical electromagnetic Schr\"odinger flows}
\author{Luca Fanelli}
\address{Luca Fanelli: Universidad del Pa\'is Vasco, Departamento de
Matem\'aticas, Apartado 644, 48080, Bilbao, Spain}
\email{luca.fanelli@ehu.es}
\author{Veronica Felli}
\address{Veronica Felli: Universit$\grave{\text{a}}$ di Milano Bicocca,
Dipartimento di Matematica e applicazioni, Via Cozzi 53, 20125, Milano, Italy%
}
\email{veronica.felli@unimib.it}
\author{Marco A. Fontelos}
\address{Marco Antonio Fontelos: ICMAT-CSIC, Ciudad Universitaria de
Cantoblanco. 28049, Madrid, Spain}
\email{marco.fontelos@icmat.es}
\author{Ana Primo}

\address{Ana Primo: ICMAT-CSIC, Ciudad Universitaria de Cantoblanco. 28049,
Madrid, Spain}
\email{ana.primo@icmat.es}

\begin{abstract}
We obtain a representation formula for solutions to Schr\"odinger  equations
with a class of homogeneous, scaling-critical electromagnetic potentials. As a consequence, we  prove the sharp $L^1\to L^\infty$ time decay estimate for the 3D-inverse square and the 2D-Aharonov-Bohm potentials.
\end{abstract}

\date{March 8, 2012}
\maketitle



\section{Introduction and statements of the results}

This work is concerned with the dispersive property of the following class
of Schr\"odinger equations with singular homogeneous electromagnetic
potentials
\begin{equation}  \label{prob}
iu_t=\left(-i\nabla+ \dfrac{{\mathbf{A}}\big(\frac{x}{|x|}\big)} {|x|}
\right)^{\!\!2} u+ \dfrac{a\big(\frac{x}{|x|}\big)}{|x|^2}\,u;
\end{equation}
here $u=u(x,t):{\mathbb{R}}^{N+1}\to{\mathbb{C}}$, $N\geq 2$, $a\in
L^{\infty}({\mathbb{S}}^{N-1}, {\mathbb{R}})$, ${\mathbb{S}}^{N-1}$ denotes
the unit $(N-1)$-dimensional sphere, and ${\mathbf{A}}\in C^1({\mathbb{S}}%
^{N-1},{\mathbb{R}}^N)$ satisfies the following  transversality condition
\begin{equation}  \label{transversality}
{\mathbf{A}}(\theta)\cdot\theta=0 \quad \text{for all }\theta\in {\mathbb{S}}%
^{N-1}.
\end{equation}
We always denote by $r:=|x|$, $\theta=x/|x|$, so that $x=r\theta$.

Equation \eqref{prob} describes the dynamics of a (non relativistic)
particle under the action of a fixed external electromagnetic field $(E,B)$
given by
\begin{equation*}
E(x):=\nabla\left(\dfrac{a\big(\frac{x}{|x|}\big)}{|x|^2}\right)\!, \qquad
B(x)=d\left(\dfrac{{\mathbf{A}}\big(\frac{x}{|x|}\big)} {|x|}\right)\!,
\end{equation*}
where $B$ is the differential of the linear 1-form associated to the vector field $\frac{{\mathbf{A}}(x/|x|)
}{|x|}$. In dimension $N=3$, due to the identification
between 1-forms and 2-forms, the magnetic field $B$ is in fact determined by
the vector field $\mathop{\rm curl}\frac{{\mathbf{A}}(x/|x|)}{|x|}$, in the
sense that
\begin{equation*}
B(x)v=\mathop{\rm curl}\big(\tfrac{{\mathbf{A}}(x/|x|)}{|x|}\big)\times v,
\qquad N=3,
\end{equation*}
for any vector $v\in{\mathbb{R}}^3$, the cross staying for the usual
vectorial product.

Under the transversality condition \eqref{transversality}, the hamiltonian
\begin{equation}  \label{eq:hamiltonian}
{\mathcal{L}}_{{\mathbf{A}},a}:= \left(-i\,\nabla+\frac{{\mathbf{A}}\big(%
\frac{x}{|x|}\big)} {|x|}\right)^{\!\!2}+\dfrac{a\big(\frac{x}{|x|}\big)}{%
|x|^2}
\end{equation}
formally acts on functions $f:{\mathbb{R}}^N\to{\mathbb{C}}$ as
\begin{equation*}
{\mathcal{L}}_{{\mathbf{A}},a}f= -\Delta  f+\frac{|{\mathbf{A}}\big(\frac{x}{%
|x|}\big)|^2+ a\big(\frac{x}{|x|}\big)-i\dive\nolimits_{{\mathbb{S}}^{N-1}}{%
\mathbf{A}}\big(\frac{x}{|x|}\big)
  }{|x|^2}\,f-2i\,\frac{{\mathbf{A}}\big(\frac{x}{|x|}\big)}{|x|}\cdot\nabla
f,
\end{equation*}
where $\dive\nolimits_{{\mathbb{S}}^{N-1}}{\mathbf{A}}$ denotes the
Riemannian  divergence of ${\mathbf{A}}$ on the unit sphere ${\mathbb{S}}%
^{N-1}$ endowed  with the standard metric.

The free Schr\"odinger equation, i.e. \eqref{prob} with ${\mathbf{A}}\equiv{%
\mathbf{0}}$ and  $a\equiv0$, can be somehow considered as the canonical
example of  dispersive equation. The unique solution $u\in\mathcal{C}({%
\mathbb{R}};L^2({\mathbb{R}}^N))$ of the Cauchy problem
\begin{equation}  \label{eq:schrofree}
\begin{cases}
iu_t=-\Delta u \\
u(x,0)=f(x)\in L^2({\mathbb{R}}^N)%
\end{cases}%
\end{equation}
can be explicitly written as follows:
\begin{equation}  \label{eq:free}
u(x,t)=e^{it\Delta}f(x):=\frac{1}{(4\pi it)^{\frac N2}} e^{i\frac{|x|^2}{4t}%
}\int_{{\mathbb{R}}^N} e^{-i\frac{x\cdot y}{2t}} e^{i\frac{|y|^2}{4t}%
}f(y)\,dy.
\end{equation}
This shows that, up to scalings and modulations, $u$ is the Fourier
transform of the initial datum $f$. Formula \eqref{eq:free} contains  most
of the relevant informations about the dispersion  which arises
along the evolution of the Schr\"odinger flow.
In dimension $N=1$,  the evolution of an initial wave packet of the form
\begin{equation*}
f_K(x)=e^{iKx}e^{-\frac{x^2}{2}}, \quad K\in\mathbb{N},
\end{equation*}
gives an important description of the phenomenon.
Inserting $f=f_K$ in \eqref{eq:free} gives in turn the
solution
\begin{equation*}
u_K(x,t)=e^{-\frac{K^2}2}u_0(x-iK,2t); \quad u_0(z,t)=\frac{1}{\sqrt{it+1}}e^{-\frac{z^2}{2(it+1)}}.
\end{equation*}
This shows that each wave travels with a speed which is proportional to
the frequency $K$, and describes both the phenomenon and the terminology of dispersion.

The above property can be quantified in
terms of a priori estimates for solutions to \eqref{eq:schrofree}. A first
consequence of \eqref{eq:free} is the time decay
\begin{equation}  \label{eq:decay}
\left\|e^{it\Delta}f(\cdot)\right\|_{L^p}\leq\frac{C}{|t|^{N\left(\frac12-%
\frac1p\right)}} \|f\|_{L^{p^{\prime }}}, \qquad p\geq2; \quad
\frac1p+\frac1{p^{\prime }}=1,
\end{equation}
for some $C=C(p,N)>0$ independent on $t$ and $f$.
In the cases $p=\infty, p=2$, \eqref{eq:decay} can be easily obtained by \eqref{eq:free} and Plancherel; the
rest of the range $2<p<\infty$ follows by Riesz-Thorin interpolation. These inequalities play
a fundamental role in many different fields, including scattering theory,
harmonic analysis and nonlinear analysis. In particular, they standardly imply the
following Strichartz estimates
\begin{equation}  \label{eq:stri}
\left\|e^{it\Delta}f\right\|_{L^p_tL^q_x}\leq C\|f\|_{L^2},
\end{equation}
for some $C>0$, where $L^p_tL^q_x:=L^p({%
\mathbb{R}};L^q({\mathbb{R}}^N))$ and the couple $(p,q)$ satisfies the
scaling condition
\begin{equation}  \label{eq:admis}
\frac2p+\frac Nq=\frac N2, \qquad p\geq2, \qquad (p,q,N)\neq(2,\infty,2).
\end{equation}
The first result in this style has been obtained by Segal in  \cite{Se} for
the wave equation; then it was generalized by  Strichartz in \cite{St} in
connection with the Restriction Theorem by  Tomas in \cite{T}. Later,
Ginibre and Velo introduced in \cite{GV1}  (see also \cite{GV2}) a different
point of view which was extensively  used by Yajima in \cite{Y1} to prove a
large amount of inequalities  for the linear Schr\"odinger equation.
Finally, Keel and Tao in  \cite{KT} completed the picture of estimates %
\eqref{eq:stri}, proving  the difficult endpoint estimate $p=2$, via
bilinear techniques, for  an abstract propagator verifying a time decay
estimate in the spirit  of \eqref{eq:decay}.

Time decay and Strichartz estimates turn out to be a  fundamental tool
in the nonlinear applications, and consequently a large literature  has been devoted, in the
last years, to obtain them in more general  situations, as for example
perturbations of the Schr\"odinger  equation with linear lower order terms,
as in \eqref{prob}. In  particular, since less regular terms usually arise
in the physically  relevant models, a deep effort has been spent in order to
overcome  the difficulty deriving from the fact that the Fourier transform
does  not fit well with differential operators with rough  coefficients. Among these, electromagnetic Schr\"odinger hamiltonians have been object of study in several papers.

 An electromagnetic Schr\"odinger equation has the form
\begin{equation}  \label{eq:schro2}
iu_t=(-i\nabla+A(x))^2u+V(x)u,
\end{equation}
where $u=u(x,t):{\mathbb{R}}^{N+1}\to{\mathbb{C}}$, $A:{\mathbb{R}}^N\to{%
\mathbb{R}}^N$, and  $V:{\mathbb{R}}^N\to{\mathbb{R}}$. Homogeneous
potentials like
\begin{equation}  \label{eq:omog}
|A|\sim\frac1{|x|}, \qquad |V|\sim\frac1{|x|^2}
\end{equation}
represent a threshold for the validity of estimates \eqref{eq:decay}
and %
\eqref{eq:stri}, as shown by Goldberg, Vega and Visciglia in
\cite{GVV}, when $A\equiv0$ and later generalized by Fanelli and
Garc\'ia in \cite{FG} for $A\neq0$ (actually, the authors in
\cite{FG,GVV} disprove Strichartz estimates, and a byproduct of this
fact is the failure of the usual time decay estimates). Notice that,
for potentials as the ones in \eqref{eq:omog}%
, equation \eqref{eq:schro2} remains invariant under the usual scaling
$%
u_\lambda(x,t)=u(x/\lambda, t/\lambda^2)$, $\lambda>0$, and this is
why we refer to it as the scaling-critical situation. We also recall
that equation \eqref{eq:schro2} is gauge invariant, namely if $u$ is a
solution to \eqref{eq:schro2}, then $v=e^{i\phi(x)}u$ solves the same
equation, with $A$ replaced by $A+\nabla\phi$ and the same magnetic
field $B$.

In the purely electric case $A\equiv0$, several authors studied the
time dispersion when the potential $V$ is close to the scaling
invariant case \eqref{eq:omog} (see \cite{B, BS, BG, DP, GeVi, GS, JN,
  JSS, PSTZ, RS} and the references therein, both for Schr\"odinger
and wave equations, and also the useful survey \cite{S}). A typical
perturbative approach consists in writing the action of the flow
$e^{it(\Delta-V)}$ via spectral theorem, and then reducing matters of
proving the desired estimate to perform a suitable analysis of the
resolvent of $-\Delta+V$, in the Agmon-H\"ormander style.  We refer to
the results by Goldberg-Schlag \cite{GS} and Rodnianski-Schlag
\cite{RS}, in which also time dependent potentials are treated, as
standard examples of this technique for Schr\"odinger equations; in
these papers, time decay estimates are obtained under integrability
conditions on $%
V$ which are close to the scaling invariant case \eqref{eq:omog}, but
do not include the critical behavior $1/|x|^2$, due to the
perturbation character of the strategy.  Another possible approach
consists in studying the mapping properties of the wave operators in
$L^p$, and obtaining the time decay for the perturbed flow
$e^{it(\Delta-V)}$ as a consequence of \eqref{eq:decay}, via
interwining properties. This point of view was introduced by Yajima in
\cite{Y2, Y3, Y4} and then followed by different authors (see
e.g. \cite%
{DF1, W1, W2}). Since it leads to a much stronger result, the
integrability conditions which are needed for the potential $V$ are
usually far from being optimal in the sense of \eqref{eq:omog}. The
unique situation in which, at our knowledge, the $L^p$-boundedness of
the wave operators is proved under almost sharp assumptions on $V$ is
the 1D-case, as it has been proven in \cite{DF1}. About Strichartz
estimates for $e^{it(\Delta-V)}$, the situation is quite clear, thanks
to the results obtained by Burq, Planchon, Stalker and Tahvildar-Zadeh
in \cite{BPSTZ1, BPSTZ}; the authors can prove a suitable
Morawetz-type estimate for the perturbed resolvent, by multiplier
techniques, which implies, together with its free countepart and free
Strichartz, the Strichartz estimates for a class of potentials $V$
which includes the ones which are critical in the sense of
\eqref{eq:omog}.

The situation in the electromagnetic case $A\neq0, V\neq0$ is quite more
complicated and weaker results are available. Some additional difficulties,
performing the above mentioned approach, come into play, due to
the introduction of a first order term in the equation, which makes more complicate the analysis of the resolvent (see e.g. \cite{DF2}%
). On the other hand, some results are available, both for estimates like %
\eqref{eq:decay} and \eqref{eq:stri}, under suitable conditions on the
potentials in \eqref{eq:schro2}, which as far as we know never permit to
recover the critical cases as in \eqref{eq:omog} (see e.g. \cite{C, CS,
EGS1, EGS2, DF2, DF3, DFVV, GST, MMT, RZ, ST, Ste} and the references
therein).

In view of the above considerations, it should be quite interesting to
produce a tool which might permit to prove the decay estimates %
\eqref{eq:decay} for equation \eqref{prob}, in which the potentials are
scaling-critical. The main goal of this manuscript is to give an explicit
representation formula for solutions to \eqref{prob}, which is in fact a
generalization of \eqref{eq:free}. In the approach we follow in the sequel,
the critical homogeneities and the transversality condition %
\eqref{transversality} play a fundamental role. We are now ready to prepare
the setting of our main results.

A key role in the representation formula we are going to derive in section %
\ref{sec:main-theor-repr} is played by the spectrum of the angular component
of the operator ${\mathcal{L}}_{{\mathbf{A}},a}$ on the unit $(N-1)$%
-dimensional sphere ${\mathbb{S}}^{N-1}$, i.e. of the operator
\begin{align}  \label{eq:angular}
L_{{\mathbf{A}},a} & =\big(-i\,\nabla_{\mathbb{S}^{N-1}}+{\mathbf{A}}\big)%
^2+a(\theta) \\
& =-\Delta_{\mathbb{S}^{N-1}}+\big(|{\mathbf{A}}|^2+ a(\theta)-i\,\dive%
\nolimits_{{\mathbb{S}}^{N-1}}{\mathbf{A}}\big)-2i\,{\mathbf{A}}\cdot\nabla_{%
\mathbb{S}^{N-1}}.  \notag
\end{align}
By classical spectral theory, $L_{{\mathbf{A}},a}$ admits a diverging
sequence of real eigenvalues with finite multiplicity $\mu_1({\mathbf{A}}%
,a)\leq\mu_2({\mathbf{A}},a)\leq\cdots\leq\mu_k({\mathbf{A}},a)\leq\cdots$,
see \cite[Lemma A.5]{FFT}. To each $k\in{\mathbb{N}}$, $k\geq 1$, we
associate a $L^{2}\big({\mathbb{S}}^{N-1},{\mathbb{C}}\big)$-normalized
eigenfunction $\psi_k$ of the operator $L_{{\mathbf{A}},a}$ on $\mathbb{S}%
^{N-1}$ corresponding to the $k$-th eigenvalue $\mu_{k}({\mathbf{A}},a)$,
i.e. satisfying
\begin{equation}  \label{angular}
\begin{cases}
L_{{\mathbf{A}},a}\psi_{k}=\mu_k({\mathbf{A}},a)\,\psi_k(\theta), & \text{in
}{\mathbb{S}}^{N-1}, \\[3pt]
\int_{{\mathbb{S}}^{N-1}}|\psi_k(\theta)|^2\,dS(\theta)=1. &
\end{cases}%
\end{equation}
In the enumeration $\mu_1({\mathbf{A}},a)\leq\mu_2({\mathbf{A}}%
,a)\leq\cdots\leq\mu_k({\mathbf{A}},a)\leq \cdots$ we repeat each eigenvalue
as many times as its multiplicity; thus exactly one eigenfunction $\psi_k$
corresponds to each index $k\in{\mathbb{N}}$. We can choose the functions $%
\psi_k$ in such a way that they form an orthonormal basis of $L^2({\mathbb{S}%
}^{N-1},{\mathbb{C}})$. We also introduce the numbers
\begin{equation}  \label{eq:alfabeta}
\alpha_k:=\frac{N-2}{2}-\sqrt{\bigg(\frac{N-2}{2}\bigg)^{\!\!2}+\mu_k({%
\mathbf{A}},a)}, \quad \beta_k:=\sqrt{\left(\frac{N-2}{2}\right)^{\!\!2}+%
\mu_k({\mathbf{A}},a)},
\end{equation}
so that $\beta_{k}=\frac{N-2}{2}-\alpha_{k}$, for $k=1,2,\dots$, which will
come into play in the sequel. Notice that $\alpha_{1}\geq\alpha_{2}\geq%
\alpha_{3}\dots$.

Under the condition
\begin{equation}  \label{eq:hardycondition}
\mu_1({\mathbf{A}},a)>-\left(\frac{N-2}{2}\right)^{\!\!2}
\end{equation}
the quadratic form associated to $\mathcal{L}_{{\mathbf{A}},a}$ is positive
definite (see Section \ref{sec:functional-setting} below and the paper \cite%
{FFT}), thus implying that the hamiltonian $\mathcal{L}_{{\mathbf{A}},a}$ is
a symmetric semi-bounded operator on $L^2({\mathbb{R}}^N;{\mathbb{C}})$
which then admits a self-adjoint extension (Friedrichs extension) with the
natural form domain. As a consequence, under assumption %
\eqref{eq:hardycondition} the unitary flow $e^{it\mathcal{L}_{{\mathbf{A}}%
,a}}$ is well defined on the domain of $\mathcal{L}_{{\mathbf{A}},a}$ by
Spectral Theorem; therefore, for every $u_0\in L^2({\mathbb{R}}^N;{\mathbb{C}%
})$, there exists a unique solution $u(\cdot,t):=e^{it\mathcal{L}_{{\mathbf{A%
}},a}}u_0\in \mathcal{C}({\mathbb{R}};L^2({\mathbb{R}}^N))$ to %
\eqref{prob} with $u(x,0)=u_0(x)$.

\begin{remark}\label{rem:t_neg}
We notice that
$$
u(\cdot,-s):=e^{-is\mathcal{L}_{{\mathbf{A
      }},a}}u_0= \overline{e^{is\mathcal{L}_{{\mathbf{-A
        }},a}}\overline{u_0}};
$$
henceforth, for the sake of simplify and without loss of generality, in the sequel we consider
$u=u(x,t):{\mathbb{R}}^{N}\times[0,+\infty)\to{\mathbb{C}}$.
\end{remark}

The main theorem of the present
paper provides a representation formula for such
solution in terms of the following  kernel
\begin{equation} \label{nucleo}
K(x,y)=\sum\limits_{k=1}^{\infty }i^{-\beta _{k}}j_{-\alpha
_{k}}(|x||y|)\psi _{k}\big(\tfrac{x}{|x|}\big)\overline{\psi _{k}\big(\tfrac{%
y}{|y|}\big)},
\end{equation}
where $\alpha _{k},\beta _{k}$ are defined in \eqref{eq:alfabeta} and, for every $\nu \in {\mathbb{R}}$,
\begin{equation*}
j_{\nu }(r):=r^{-\frac{N-2}{2}}J_{\nu +\frac{N-2}{2}}(r)
\end{equation*}%
with $J_{\nu }$ denoting the Bessel function of the first kind
\begin{equation*}
  J_{\nu }(t)=\bigg(\frac{t}{2}\bigg)^{\!\!\nu }\sum\limits_{k=0}^{\infty }
  \dfrac{(-1)^{k}}{\Gamma (k+1)\Gamma (k+\nu +1)}\bigg(\frac{t}{2}\bigg)
  ^{\!\!2k}.
\end{equation*}
The following lemma provides uniform convergence on compacts of the
queue of the series in \eqref{nucleo}.
\begin{Lemma}\label{l:queue}
 Let $a\in L^{\infty }({\mathbb{S}}^{N-1},{\mathbb{R}})$, ${%
\mathbf{A}}\in C^{1}({\mathbb{S}}^{N-1},{\mathbb{R}}^{N})$ such that
\eqref{transversality} and \eqref{eq:hardycondition} hold. Then
there exists $k_0\geq1$ such that the series
$$
\sum_{k=k_0+1}^{\infty }i^{-\beta _{k}}j_{-\alpha
_{k}}(|x||y|)\psi _{k}\big(\tfrac{x}{|x|}\big)\overline{\psi _{k}\big(\tfrac{%
y}{|y|}\big)}
$$
is uniformly convergent on compacts and
$$
K(x,y)-\sum_{k=1}^{k_0}i^{-\beta _{k}}j_{-\alpha _{k}}(|x||y|)\psi
_{k}\big(\tfrac{x}{|x|}\big)\overline{\psi _{k}\big(\tfrac{%
    y}{|y|}\big)}\in L_{\rm loc}^{\infty }(\R^{2N},{\mathbb{C}}).
$$
\end{Lemma}

\noindent We can now state the main result of this
paper.
\begin{Theorem}[Representation formula]\label{Main}
  Let $a\in L^{\infty }({\mathbb{S}}^{N-1},{\mathbb{R}})$,
  ${\mathbf{A}}\in C^{1}({\mathbb{S}}^{N-1},{\mathbb{R}}^{N})$ such
  that \eqref{transversality} and \eqref{eq:hardycondition} hold. Let
  $\mathcal{L}_{ {\mathbf{A}},a}$ as in \eqref{eq:hamiltonian} and
  $K$ as in \eqref{nucleo}.
If $u_{0}\in L^{2}({\mathbb{R}}^{N})$ and $u(x,t)=e^{it\mathcal{L}_{{\mathbf{%
A}},a}}u_{0}(x)$, then, for all $t>0$,
\begin{equation}\label{representation}
u(x,t)=\frac{e^{\frac{i|x|^{2}}{4t}}}{i(2t)^{{N}/{2}}}\int_{{\mathbb{R}}%
^{N}}K\bigg(\frac{x}{\sqrt{2t}},\frac{y}{\sqrt{2t}}\bigg)e^{i\frac{|y|^{2}}{%
4t}}u_{0}(y)\,dy.
\end{equation}
\end{Theorem}

\begin{remark}\label{rem:int}
  The integral at the right hand side of \eqref{representation} is
  understood in the sense the improper multiple integrals, i.e.
\begin{equation*}
  u(x,t)=\frac{e^{\frac{i|x|^{2}}{4t}}}{i(2t)^{{N}/{2}}}
  \lim_{R\to+\infty}\int_{B_R}K\bigg(\frac{x}{\sqrt{2t}},\frac{y}{\sqrt{2t}}\bigg)e^{i\frac{|y|^{2}}{%
      4t}}u_{0}(y)\,dy,
\end{equation*}
where $B_R:=\{y\in\R^N:|y|<R\}$.
\end{remark}

\begin{remark}
\label{rem:free_case} Formula \eqref{representation} is in fact a
generalization of \eqref{eq:free}. Indeed, in the free case, i.e. ${\mathbf{A%
}}\equiv {\mathbf{0}}$ and $a\equiv 0$, the operator $L_{{\mathbf{A}},a}$
reduces to the Laplace Beltrami operator $-\Delta _{\mathbb{S}^{N-1}}$,
whose eigenvalues are given by
\begin{equation*}
\lambda _{\ell }=(N-2+\ell )\ell ,\quad \ell =0,1,2,\dots ,
\end{equation*}%
having the $\ell $-th eigenvalue $\lambda _{\ell }$ multiplicity
\begin{equation*}
m_{\ell }=\frac{(N-3+\ell )!(N+2\ell -2)}{\ell !(N-2)!},
\end{equation*}%
and whose eigenfunctions coincide with the usual spherical harmonics. For
every $\ell \geq 0$, let $\{Y_{\ell ,m}\}_{m=1,2,\dots ,m_{\ell }}$ be a $%
L^{2}(\mathbb{S}^{N-1},{\mathbb{C}}^{N})$-orthonormal basis of the
eigenspace of $-\Delta _{\mathbb{S}^{N-1}}$ associated to $\lambda _{\ell }$
with $Y_{\ell ,m}$ being spherical harmonics of degree $\ell $. Hence we
have that
\begin{align*}
& \text{if }k=1,\text{ then }\mu _{1}({\mathbf{0}},0)=\lambda _{0}=0,\quad
\alpha _{1}=0,\quad \beta _{1}=\tfrac{N-2}{2}, \\
& \text{if }k>1\text{ and }{\textstyle{\sum_{n=0}^{\ell -1}}}m_{n}<k\leq {%
\textstyle{\sum_{n=0}^{\ell }}}m_{n},\text{ then }%
\begin{cases}
\mu _{k}({\mathbf{0}},0)=\lambda _{\ell } \\
\alpha _{k}=-\ell  \\
\beta _{k}=\frac{N-2}{2}+\ell
\end{cases}%
, \\
& \{\psi _{k}\}_{k=1,2,\dots }=\{Y_{\ell ,m}\}_{\substack{ \ell =1,2,\dots
\quad \ \  \\ m=1,2,\dots ,m_{\ell }}}.
\end{align*}%
The Jacobi-Anger expansion for plane waves combined with the Addition
Theorem for spherical harmonics (see for example \cite[formula (4.8.3), p.
116]{Ismail} and \cite[Corollary 1]{BeStr}) yields
\begin{equation*}
e^{ix\cdot y}=(2\pi )^{N/2}\big(|x||y|\big)^{-\frac{N-2}{2}}\sum_{\ell
=0}^{\infty }i^{\ell }J_{\ell +\frac{N-2}{2}}\big(|x||y|\big)\bigg(%
\sum_{m=1}^{m_{\ell }}Y_{\ell ,m}\big(\tfrac{x}{|x|}\big)\overline{%
Y_{\ell ,m}\big(\tfrac{y}{|y|}\big)}\bigg)
\end{equation*}%
for all $x,y\in {\mathbb{R}}^{N}$. Then in the free case ${\mathbf{A}}\equiv
{\mathbf{0}}$, $a\equiv 0$, we have that
\begin{equation*}
K(x,y)=\frac{e^{-ix\cdot y}}{(2\pi )^{\frac{N}{2}}i^{\frac{N-2}{2}}},
\end{equation*}%
which, together with \eqref{representation} and taking into account
that if $t=-s$, $s>0,$ then
$u(\cdot,-s)=\overline{e^{is(-\Delta)}\overline{u_0}}$, gives in turn
\eqref{eq:free}.
\end{remark}

\begin{remark}
\label{rem:eq_kernel}  We remark that, for every $y\in{\mathbb{R}}^N$ fixed,
the function $K(\cdot,y)$  formally solves the equation
\begin{equation*}
\mathcal{L}_{{\mathbf{A}},a}K(\cdot,y)=|y|^2K(\cdot,y),
\end{equation*}
as one can easily check; in fact this fits with the free case, in which $K$
is the plane wave $e^{-ix\cdot y}$, up to constants.
\end{remark}

Formula \eqref{representation} is not present in the literature, at our
knowledge; moreover, as far as we understand, it should provide a
fundamental tool for several different applications. A first immediate
consequence of the representation formula \eqref{representation} is the
following corollary.

\begin{Corollary}[Time decay]
\label{cor:decay}  Let $a\in L^{\infty}({\mathbb{S}}^{N-1},{\mathbb{R}})$, ${%
\mathbf{A}}\in C^1({\mathbb{S}}^{N-1},{\mathbb{R}}^N)$ such that %
\eqref{transversality} and  \eqref{eq:hardycondition} hold, and $\mathcal{L}%
_{{\mathbf{A}},a}$ as in  \eqref{eq:hamiltonian}. If
\begin{equation}  \label{eq:claim}
\sup_{x,y\in{\mathbb{R}}^N}|K(x,y)|<+\infty,
\end{equation}
with $K$ as in \eqref{nucleo}, then the following estimate holds
\begin{equation}  \label{eq:decaygen}
\left\|e^{it\mathcal{L}_{{\mathbf{A}},a}}f(\cdot)\right\|_{L^p}\leq \frac{C}{%
|t|^{N\left(\frac12-\frac1p\right)}}\|f\|_{L^{p^{\prime }}}, \quad p\in[%
2,+\infty], \quad \frac1p+\frac1{p^{\prime }}=1,
\end{equation}
for some $C=C({\mathbf{A}}, a, p)>0$ which does not depend on $t$ and $f$.
\end{Corollary}

\begin{proof}
  The proof is quite immediate. Formula
  \eqref{representation}, \eqref{eq:claim}, and Remark \ref{rem:t_neg} automatically yield
  \eqref{eq:decaygen} in the case $p=\infty$.  The rest of the range
  $p\geq2$ in \eqref{eq:decaygen} then follows by interpolation with
  the $L^2$ conservation.
\end{proof}

\begin{remark}
\label{rem:fundamental}  Once the matters to prove a time decay estimate are
reduced to the  study of the kernel $K$ in \eqref{nucleo}, the behavior of
the  spherical Bessel functions $j_{-\alpha_k}$ comes into play. The
crucial fact to notice is that condition \eqref{eq:claim} is  strictly
related to the requirement $\alpha_1\leq0$ which in other  words, by
\eqref{eq:alfabeta} means $\mu_1(A,a)\geq0$. It is easy to verify that
\eqref{eq:claim} implies that $\mu_1(A,a)\geq0$, while arguing as in
the proof of Lemma \ref{l:queue} we can easily check that $\mu_1(A,a)\geq0$ implies that $K$ is
locally bounded.

We are  strongly motivated by
the examples in the sequel to conjecture that  in fact conditions 
\eqref{eq:claim} and $\mu_1(A,a)\geq0$ are  equivalent.
\end{remark}

We now pass to give a couple of relevant examples in which the abstract
assumption \eqref{eq:claim} can be checked by hands, and the optimal time
decay can be obtained by working directly on the representation formula 
\eqref{representation}.

\subsection{Application 1: Aharonov-Bohm field}

We start with a 2D example of purely magnetic field, which is given by
potentials associated to thin solenoids: if the radius of the solenoid tends
to zero while the flux through it remains constant, then the particle is
subject to a $\delta$-type magnetic field, which is called \emph{%
Aharonov-Bohm} field. A vector potential associated to the Aharonov-Bohm
magnetic field in ${\mathbb{R}}^2$ has the form
\begin{equation}  \label{eq:Bohm}
{\boldsymbol{\mathcal{A}}}(x_1,x_2)=\alpha\bigg(-\frac{x_2}{|x|^2},\frac{x_1%
}{|x|^2}\bigg),\quad (x_1,x_2)\in{\mathbb{R}}^2,
\end{equation}
with $\alpha\in{\mathbb{R}}$ representing the circulation of ${\boldsymbol{%
\mathcal{A}}}$ around the solenoid. Notice that the potential in %
\eqref{eq:Bohm} is singular at $x=0$, homogeneous of degree $-1$ and
satisfies the transversality condition \eqref{transversality}.

This situation corresponds to problem \eqref{prob} with
\begin{equation*}
N=2, \quad {\mathbf{A}}(\theta)={\mathbf{A}}(\cos t,\sin t)=\alpha(-\sin
t,\cos t), \quad a(\theta)=0,
\end{equation*}
so that equation \eqref{prob} takes the form
\begin{equation}  \label{eq:AB}
iu_{t}= \left(-i\,\nabla+\alpha\bigg(-\frac{x_2}{|x|^2},\frac{x_1}{|x|^2} %
\bigg)\right)^{\!\!2}u,
\end{equation}
with $x=(x_1,x_2)\in {\mathbb{R}}^2$. In this case, an explicit calculation
yields
\begin{equation*}
\{\mu_k({\mathbf{A}},0):k\in{\mathbb{N}}\setminus\{0\}\}=\{(\alpha-j)^2:j\in{%
\mathbb{Z}}\},
\end{equation*}
(see e.g. \cite{lw} for details), and in particular
\begin{equation*}
\mu_1({\mathbf{A}},0)=\big(\mathop{\rm dist}(\alpha,{\mathbb{Z}})\big)%
^2\geq0.
\end{equation*}
If $\mathop{\rm dist}(\alpha,{\mathbb{Z}})\neq\frac12$, then all the
eigenvalues are simple and the eigenspace associated to the eigenvalue $%
(\alpha-j)^2$ is generated by $\psi(\cos t,\sin t)=e^{-ijt}$. If $%
\mathop{\rm dist}(\alpha,{\mathbb{Z}})=\frac12$, then all the eigenvalues
have multiplicity $2$. The following result is an interesting consequence of
Theorem \ref{Main}.

\begin{Theorem}[Time decay for Aharonov-Bohm]\label{thm:AB}  
Let $N=2$, $\alpha\in{\mathbb{R}}$, and define
\begin{equation*}
\mathcal{L}_{\alpha}:=\left(-i\,\nabla+\alpha\bigg(-\frac{x_2}{|x|^2},\frac{%
x_1}{|x|^2} \bigg)\right)^{\!\!2}.
\end{equation*}
Then the following estimate holds
\begin{equation} \label{eq:decayAB}
  \left\|e^{it\mathcal{L}_{\alpha}}f(\cdot)\right\|_{L^p}\leq
  \frac{C}{ |t|^{2\left(\frac12-\frac1p\right)}}\|f\|_{L^{p^{\prime
      }}}, \quad p\in[2,+\infty], \quad \frac1p+\frac1{p^{\prime }}=1,
\end{equation}
for some constant $C=C(\alpha, p)>0$ which does not depend on $t$ and $f$.
\end{Theorem}

\begin{remark}
Since  $\text{curl}\,A(x)\equiv0$ if $x\neq0$, the action of the magnetic
field in the  Aharonov-Bohm case is concentrated at the origin. However the
potential ${\mathbf{A}}$ cannot be  eliminated by gauge transformations;
this is in fact the peculiar  property of Aharonov-Bohm fields, which indeed
describe an  interesting difference between the classical and quantum
version of  the electromagnetic theory. Due to the above remark, estimates  %
\eqref{eq:decayAB} are not trivial, as far as we understand, and at  our
knowledge they are not known. We finally stress that the  algebraic
structure of Aharonov-Bohm potentials is exactly the one  which has been
used in \cite{FG} in order to disprove the dispersion  (in that case
Strichartz inequalities) in the case of magnetic field  which decay less
than $|x|^{-1}$ at infinity, in dimension  $N\geq3$. In 2D, counterexamples
as the ones in \cite{FG} are  missing, and it is still unclear what might
happen for potentials with  less decay than the one in \eqref{eq:AB}.
\end{remark}

\subsection{Application 2: The inverse square potential}

We also present an application of formula \eqref{representation} in the case
of perturbation of the Laplace operator in dimension $N=3$ with an inverse
square electric potential; more precisely, we consider problem \eqref{prob}
with ${\mathbf{A}}=0$ and $a(\theta)=a=\mbox{constant}$, so that the
hamiltonian \eqref{eq:hamiltonian} takes the form
\begin{equation}  \label{eq:inverse}
\mathcal{L}_{a}:=-\Delta +\frac{a}{|x|^2},\quad\text{in }{\mathbb{R}}^3,
\end{equation}
and condition (\ref{eq:hardycondition}) reads as $a>-\frac14$. Since in this
case the angular eigenvalue problem \eqref{eq:angular} becomes
\begin{equation}  \label{Y}
\left\{
\begin{array}{l}
-\Delta_{\mathbb{S}^{2}}\psi_{k}=(\mu_{k}({\mathbf{0}},a)-a) \psi_{k} ,\text{
in } \mathbb{S} ^{2}, \\
\|\psi_{k}\|_{L^{2}(\mathbb{S}^{2})}=1,%
\end{array}%
\right.
\end{equation}
we have that $\{\psi_{k}\}_{k=1}^{\infty}$ are the well known spherical
harmonics and
\begin{equation*}
\mu_{k}({\mathbf{0}},a)=a + \mu_{k}({\mathbf{0}},0)\,\,.
\end{equation*}
As in Remark \ref{rem:free_case},  for every $\ell\geq0$, let $%
\{Y_{\ell,m}\}_{m=1,2,\dots,2\ell+1}$ be a $L^2(\mathbb{S}^{N-1},{\mathbb{C}}%
^N)$-orthonormal basis for the eigenspace of $-\Delta_{\mathbb{S}^{2}}$
associated to the $\ell$-th eigenvalue $\lambda_\ell=\ell(\ell+1)$  of $%
-\Delta_{\mathbb{S}^{2}}$ (which has multiplicity $m_\ell=2\ell+1$), with $%
Y_{\ell,m}$ being spherical harmonics of degree $\ell$. Hence we have that
\begin{align}\label{eq:alfa1}
&\text{if }k=1, \text{ then }
  \mu_1({\mathbf{0}},a)=a,\quad\alpha_1=\tfrac12-%
  \sqrt{\tfrac14+a}, \\
\label{eq:alpha>1}&\text{if } k>1\text{ and } \ell^2<k\leq (\ell+1)^2, \text{ then }
\begin{cases}
\mu_k({\mathbf{0}},a)=a+\ell(\ell+1)= \\
\alpha_k=\frac12-\sqrt{\big(\ell+\tfrac12\big)^2+a}%
\end{cases}%
.
\end{align}
Notice that $\alpha_1\leq 0$ if and only if $a\geq0$. We define the well
known zonal functions
\begin{equation}  \label{eq:zonal}
Z_{\theta }^{(\ell)}( \theta^{\prime })=\sum _{m=1}^{2\ell+1}
Y_{\ell,m}(\theta) \overline{Y_{\ell,m}(\theta^{\prime })},\quad
\theta,\theta^{\prime }\in {\mathbb{S}}^2, \quad \ell=0,1,2,\dots.
\end{equation}
The study of the kernel in \eqref{nucleo}, which can be rewritten as
\begin{equation}  \label{eq:Sinverse}
K(x,y)=\sum_{\ell=0}^{\infty }i^{-\sqrt{(\ell+1/2)^2+a}}\,j_{ -\frac12+\sqrt{%
(\ell+1/2)^2+a}} \big(|x||y|\big)Z_{\frac{x}{|x|} }^{(\ell)}\Big( \tfrac{y}{%
|y|} \Big),
\end{equation}
permits to prove the following result.

\begin{Theorem}[Time decay for inverse square potentials]\label{thm:inversesquare}
Let $N=3$, $a>-\frac14$, and define $\mathcal{L}_a$ by \eqref{eq:inverse}.
\begin{itemize}
\item[i)] If $a\geq0$, then the following estimates hold
\begin{equation}  \label{eq:decayinverse}
\left\|e^{it\mathcal{L}_{a}}f(\cdot)\right\|_{L^p}\leq \frac{C}{
|t|^{3\left(\frac12-\frac1p\right)}}\|f\|_{L^{p^{\prime }}}, \quad p\in[
2,+\infty], \quad \frac1p+\frac1{p^{\prime }}=1,
\end{equation}
for some constant $C=C(a, p)>0$ which does not depend on $t$ and $f$.

\item[ii)] If $-\frac14<a<0$, let  $\alpha_1$ as in (\ref{eq:alfa1}), and define
\begin{equation*}
\|u\|_{p,\alpha_1}:=\bigg(\int_{{\mathbb{R}}^3}(1+|x|^{-\alpha_1})^{2-p}|u(x)|^p\,dx\bigg)^{\!1/p}, \quad
p\geq1.
\end{equation*}
Then the following estimates hold
\begin{equation}  \label{eq:decayinverse2}
\left\|e^{it\mathcal{L}_{a}}f(\cdot)\right\|_{p,\alpha_1}\leq \frac{%
C(1+|t|^{\alpha_1})^{1-\frac{2}{p}}}{|t|^{3\left(\frac12-\frac1p\right)}}%
\|f\|_{p',\alpha_1}, \quad p\geq2, \quad \frac1p+\frac1{p^{\prime
}}=1,
\end{equation}
for some constant $C=C(a, p)>0$ which does not depend on $t$ and $f$.
\end{itemize}
\end{Theorem}

\begin{remark}
As far as we know, the best dispersive results concerning this kind  of
operators are about Strichartz estimates, and have been obtained  by Burq,
Planchon, Stalker and Tahvildar-Zadeh in  \cite{BPSTZ1,BPSTZ}. As a fact,
estimates \eqref{eq:decayinverse}  imply the ones obtained in \cite{BPSTZ1},
by the standard  Ginibre-Velo and Keel-Tao techniques in \cite{GV1, KT}. On
the other  hand, in \cite{BPSTZ} the authors can treat more general
potentials with critical decay, including e.g. the cases in  which $%
a=a(x/|x|)$ is a 0-degree homogeneous function; in addition,  we think that
the restriction $N=3$ in Theorem  \ref{thm:inversesquare} is not in fact a
relevant obstruction. We  are motivated to claim that a deeper analysis of
formula  \eqref{representation} should permit to prove the analog to Theorem
\ref{thm:inversesquare} in the more general case $a=a(x/|x|)$, but  this
will not be treated in the present paper.

Moreover, notice that $\alpha_1>0$, in the range $-\frac14<a<0$, see  %
\eqref{eq:alfa1}, so that the decay in \eqref{eq:decayinverse2} is  weaker
than the usual one. We find it an interesting phenomenon,  since on the
other hand the usual Strichartz estimates are still  true in this range, as
proved in \cite{BPSTZ1}. Estimates \eqref{eq:decayinverse2} are presumably sharp and, at our
knowledge, new.
\end{remark}

The rest of the paper is organized as follows. In Section \ref%
{sec:functional-setting}, we describe the functional setting in which we
work, in order to prepare the proof of the main result, Theorem \ref{Main};
Section \ref{sec:spectr-harm-magn} is then devoted to the study of the
spectral properties of a magnetic harmonic oscillator with inverse square
potential, denoted by $T_{{\mathbf{A}},a}$ (see formula \eqref{operator}),
which comes into play when a suitable ansatz (formula \eqref{varphi}) is
stated; finally, Section \ref{sec:main-theor-repr} is devoted to the proof
of Theorem \ref{Main}, while in the last Sections \ref{sec:ahar-bohm-magn}
and \ref{sec:inverse} we prove the applications, Theorems \ref{thm:AB} and %
\ref{thm:inversesquare}.

\section{Functional setting}

\label{sec:functional-setting} Let us define the following Hilbert spaces:

\begin{itemize}
\item the space ${\mathcal{H}}$ as the completion of $C^{\infty}_{\mathrm{c}%
}({\mathbb{R}}^N\setminus\{0\},{\mathbb{C}})$ with respect to the norm
\begin{equation*}
\|\phi\|_{{\mathcal{H}}}=\bigg(\int_{{\mathbb{R}}^N}\bigg(|\nabla\phi(x)|^2+
\Big(|x|^2+\frac1{|x|^2}\Big)|\phi(x)|^2\bigg) \,dx\bigg)^{\!\!1/2};
\end{equation*}

\item {the space $H$} as the completion of $C^{\infty}_{\mathrm{c}}({\mathbb{%
R}}^N,{\mathbb{C}})$ with respect to the norm
\begin{equation*}
\|\phi\|_{H}=\bigg(\int_{{\mathbb{R}}^N}\Big(|\nabla\phi(x)|^2+  \big(|x|^2+1%
\big)|\phi(x)|^2\Big) \,dx\bigg)^{\!\!1/2};
\end{equation*}

\item {the space ${\mathcal{H}}_{{\mathbf{A}}}$} as the completion of $%
C^{\infty}_{\mathrm{c}}({\mathbb{R}}^N\setminus\{0\},{\mathbb{C}})$ with
respect to the norm
\begin{equation*}
\|\phi\|_{{\mathcal{H}}_{\mathbf{A}}}=\bigg(\int_{{\mathbb{R}}^N}\Big(%
|\nabla_{\mathbf{A}}\phi(x)|^2+  \big(|x|^2+1\big)|\phi(x)|^2\Big) \,dx\bigg)%
^{\!\!1/2}
\end{equation*}
with $\nabla_{{\mathbf{A}}}\phi= \nabla\phi+i\,\frac {{\mathbf{A}}(x/|x|)}{%
|x|}\phi$.
\end{itemize}
From the above definition, it follows immediately that
\begin{equation}  \label{eq:cont_emb}
\mathcal{H }\hookrightarrow H\quad\text{with continuous embedding}.
\end{equation}
A further comparison between the above defined spaces can be derived from
the well known diamagnetic inequality (see e.g. \cite{LL})
\begin{equation}  \label{eq:diamagnetic}
|\nabla |\phi|(x)|\leq \left|\nabla \phi(x)+i\frac{{\mathbf{A}}(x/|x|)}{|x|}%
\phi(x)\right|, \qquad N\geq2,
\end{equation}
which holds for a.e. $x\in{\mathbb{R}}^N$ and for all $\phi\in C^{\infty}_{%
\mathrm{c}}({\mathbb{R}}^N\setminus\{0\},{\mathbb{C}})$, and the classical
Hardy inequality (see e.g. \cite{GP,HLP})
\begin{equation}  \label{eq:hardy}
\int_{{\mathbb{R}}^N}|\nabla \phi(x)|^{2}\,dx\geq \bigg(\frac{N-2}{2}\bigg)%
^{\!\!2}\int_{{\mathbb{R}}^N} \frac{|\phi(x)|^2}{|x|^{2}}\,dx,
\end{equation}
which holds for all $\phi\in\mathcal{C}_{\mathrm{c}}^\infty({\mathbb{R}}^N,{%
\mathbb{C}})$ and $N\ge 3$. We notice that the presence of a vector
potential satisfying a suitable non-degeneracy condition allows to recover a
Hardy inequality even for $N=2$. Indeed, if $N=2$, \eqref{transversality}
holds, and
\begin{equation}  \label{eq:circuit}
\Phi_{\mathbf{A}}:=\frac1{2\pi}\int_0^{2\pi}\alpha(t)\,dt \not\in{\mathbb{Z}}%
,\quad \text{where }\alpha(t):={\mathbf{A}}(\cos t,\sin t)\cdot(-\sin t,\cos
t),
\end{equation}
then functions in $C^{\infty}_{\mathrm{c}}({\mathbb{R}}^N\setminus\{0\},{%
\mathbb{C}})$ satisfy the following Hardy inequality
\begin{equation}  \label{eq:hardyN2}
\Big(\min_{k\in{\mathbb{Z}}}|k-\Phi_{\mathbf{A}}|\Big)^2\int_{{\mathbb{R}}^2}%
\frac{|u(x)|^2}{|x|^2}\,dx \leq \int_{{\mathbb{R}}^2}\bigg|\nabla u(x)+i\,
\frac{{\mathbf{A}}\big({x}/{|x|}\big)} {|x|}\,u(x)\bigg|^2\,dx
\end{equation}
being $\big(\min_{k\in{\mathbb{Z}}}|k-\Phi_{\mathbf{A}}|\big)^2$ the best
constant, as proved in \cite{lw}.

Combining (\ref{eq:hardy}), (\ref{eq:diamagnetic}), and (\ref{eq:hardyN2}),
it is easy to verify that if $N\geq 3$, then $H=\mathcal{H}={\mathcal{H}}_{%
\mathbf{A}}$, being the norms $\|\cdot\|_{H}$, $\|\cdot\|_{{\mathcal{H}}}$
and $\|\cdot\|_{{\mathcal{H}}_{\mathbf{A}}}$ equivalent. If $N=2$ then $%
\mathcal{H}\varsubsetneq H$; on the other hand, if $N=2$ and (\ref%
{transversality}), (\ref{eq:circuit}) hold, from (\ref{eq:diamagnetic}) and (%
\ref{eq:hardyN2}) we deduce that $\mathcal{H}={\mathcal{H}}_{\mathbf{A}}$,
being the norms $\|\cdot\|_{{\mathcal{H}}}$, $\|\cdot\|_{{\mathcal{H}}_{%
\mathbf{A}}}$ equivalent.

From (\ref{eq:cont_emb}) and \cite[Proposition 6.1]{KW}, we also deduce that
\begin{equation}  \label{eq:compact_emb}
{\mathcal{H}} \text{ is compactly embedded into }L^{p}({\mathbb{R}}^{N})
\end{equation}
for all
\begin{equation*}
2\leq p<
\begin{cases}
2^{*}=\frac{2N}{N-2}, & \text{if }N\geq3, \\
+\infty, & \text{if }N=2.%
\end{cases}
\end{equation*}
The quadratic form $Q_{{\mathbf{A}},a}$ associated to ${\mathcal{L}}_{{%
\mathbf{A}},a}$, i.e.
\begin{align}  \label{eq:qf}
&Q_{{\mathbf{A}},a}:{\mathcal{D}}^{1,2}_{*}({\mathbb{R}}^N,{\mathbb{C}})\to{%
\mathbb{R}}, \\
& Q_{{\mathbf{A}},a}(\phi):=\int_{{\mathbb{R}}^N} \bigg[ \big| \nabla_{%
\mathbf{A}}\phi(x)\big|^2 -\frac{a\big({x}/{|x|}\big)}{|x|^2}|\phi(x)|^2%
\bigg]\,dx,  \notag
\end{align}
with ${\mathcal{D}}^{1,2}_{*}({\mathbb{R}}^N,{\mathbb{C}})$ being the
completion of $C^\infty_{\mathrm{c}}({\mathbb{R}}^N\setminus\{0\},{\mathbb{C}%
})$ with respect to the norm
\begin{equation*}
\|u\|_{{\mathcal{D}}^{1,2}_{*}({\mathbb{R}}^N,{\mathbb{C}})}:=\bigg(\int_{{%
\mathbb{R}}^N}\bigg(\big|\nabla u(x)\big|^2+ \frac{|u(x)|^2}{|x|^2}\bigg) %
\,dx\bigg)^{\!\!1/2},
\end{equation*}
is positive definite if and only if (\ref{eq:hardycondition}) holds, see
\cite[Lemma 2.2]{FFT}. In particular, assumption (\ref{eq:hardycondition})
ensures that the operator ${\mathcal{L}}_{{\mathbf{A}},a}$ is semibounded
from below, self-adjoint on $L^2$ with the natural form domain, and that
there exists some constant $C(N,{\mathbf{A}},a)>0$ such that
\begin{equation}  \label{eq:posdef}
\int_{{\mathbb{R}}^N} \bigg[ \big| \nabla_{\mathbf{A}}\phi(x)\big|^2 -\frac{a%
\big({x}/{|x|}\big)}{|x|^2}|\phi(x)|^2+\frac{|x|^2}4|\phi(x)|^2 \bigg]\,dx
\geq C(N,{\mathbf{A}},a)\|\phi\|_{\mathcal{H}}^2,
\end{equation}
for all $\phi\in\mathcal{H}$ (see \cite{FFT}).

Up to a \textit{pseudo conformal} change of variable, see \cite{KW},
equation (\ref{prob}) can be rewritten in terms of a quantum harmonic
oscillator with the singular electromagnetic potential, as stated in the
following lemma.

\begin{Lemma}
Let (\ref{eq:hardycondition}) hold and $u\in C({\mathbb{R}}; L^{2}({\mathbb{R%
}}^N))$ be  a solution to \eqref{prob}. Then
\begin{equation}  \label{varphi}
\varphi(x,t)= (1+t^2)^{\frac{N}{4}}u\big(\sqrt{1+t^2}x,t\big)e^{-it\frac{%
|x|^2}{4}}
\end{equation}
satisfies
\begin{align*}
&\varphi\in C({\mathbb{R}}; L^{2}({\mathbb{R}}^N)),\quad \varphi (x, 0)=
u(x,0), \\
&\|\varphi(\cdot,t)\|_{L^{2}({\mathbb{R}}^N)}=\|u(\cdot,t)\|_{L^{2}({\mathbb{%
R}}^N)} \text{ for all }t\in{\mathbb{R}},
\end{align*}
and
\begin{equation}  \label{varphieq}
i\dfrac{d\varphi}{d t}(x,t)= \dfrac{1}{(1+t^2)} \bigg({\mathcal{L}}_{{%
\mathbf{A}},a}\varphi(x,t)+\frac{1}{4}|x|^2 \varphi(x,t)\bigg).
\end{equation}
\end{Lemma}

A representation formula for solutions $u$ to (\ref{prob}) can be found by
expanding the transformed solution $\varphi$ to (\ref{varphieq}) in Fourier
series with respect to an orthonormal basis of $L^2({\mathbb{R}}^N)$
consisting of eigenfunctions of the following quantum harmonic oscillator
operator perturbed with singular homogeneous electromagnetic potentials
\begin{equation}  \label{operator}
T_{{\mathbf{A}},a}:{\mathcal{H}}\to {\mathcal{H}}^\star,\quad T_{{\mathbf{A}}%
,a}={\mathcal{L}}_{{\mathbf{A}},a}+\frac{1}{4}|x|^2
\end{equation}
acting as
\begin{multline}  \label{operator2}
{}_{{\mathcal{H}}^\star}\langle T_{{\mathbf{A}},a}v,w \rangle_{{\mathcal{H}}}
\\
= \int_{{\mathbb{R}}^N}\bigg(\nabla_{{\mathbf{A}}} v(x)\cdot\overline{%
\nabla_{{\mathbf{A}}} w(x)}-\frac{a(\frac{x}{|x|})}{|x|^2}\,v(x)\overline{%
w(x)} +\frac{|x|^2}{4} v(x) \overline{w(x)}\bigg)\,dx,
\end{multline}
for all $v,w\in{\mathcal{H}}$, where ${\mathcal{H}}^\star$ denotes the dual
space of ${\mathcal{H}}$ and ${}_{{\mathcal{H}}^\star}\langle
\cdot,\cdot\rangle_{{\mathcal{H}}}$ is the corresponding duality product.

\section{The spectrum of $T_{{\mathbf{A}},a}$}

\label{sec:spectr-harm-magn} From (\ref{eq:compact_emb}), (\ref{eq:posdef}),
and classical spectral theory, we can easily deduce the following abstract
description of the spectrum of $T_{{\mathbf{A}},a}$.

\begin{Lemma}
\label{Hilbert}  Let ${\mathbf{A}}\in C^1({\mathbb{S}}^{N-1},{\mathbb{R}}^N)$
and $a\in  L^{\infty}\big({\mathbb{S}}^{N-1}\big)$ such that %
\eqref{eq:hardycondition}  holds. Then the spectrum of the operator $T_{{%
\mathbf{A}},a}$ defined in  (\ref{operator}--\ref{operator2}) consists of a
diverging sequence of real  eigenvalues with finite multiplicity. Moreover,
there exists an  orthonormal basis of $L^{2}({\mathbb{R}}^{N})$ whose
elements belong to  ${\mathcal{H}}$ and are eigenfunctions of $T_{{\mathbf{A}%
},a}$.
\end{Lemma}

The following proposition gives a complete description of the spectrum of
the operator $T_{{\mathbf{A}},a}$.

\begin{Proposition}
\label{spectrum} The set of the eigenvalues of the operator $T_{{\mathbf{A}}%
,a}$ is
\begin{equation*}
\big\{ \gamma_{m,k}: k,m\in{\mathbb{N}}, k\geq 1\big\}
\end{equation*}
where
\begin{equation}  \label{eigenvalues}
\gamma_{m,k}=2m-\alpha_k+\dfrac N2, \quad \alpha_k=\frac{N-2}{2}-\sqrt{\bigg(%
\frac{N-2}{2}\bigg)^{\!\!2}+\mu_k({\mathbf{A}},a)},
\end{equation}
and $\mu_k({\mathbf{A}},a)$ is the $k$-th eigenvalue of the operator $L_{{%
\mathbf{A}},a}$ on the sphere $\mathbb{S}^{N-1}$. Each eigenvalue $%
\gamma_{m,k}$ has finite multiplicity equal to
\begin{equation*}
\#\bigg\{j\in{\mathbb{N}},j\geq 1: \frac{\gamma_{m,k}}{2}+\frac{\alpha_j}%
2-\frac N4\in{\mathbb{N}}\bigg\}
\end{equation*}
and a basis of the corresponding eigenspace is
\begin{equation*}
\left\{V_{n,j}: j,n\in{\mathbb{N}},j\geq  1,\gamma_{m,k}=2n-\alpha_j+\frac
N2 \right\},
\end{equation*}
where
\begin{equation}  \label{eigenvectors}
V_{n,j}(x)= |x|^{-\alpha_j}e^{-\frac{|x|^2}{4}}P_{j,n}\Big(\frac{|x|^2}{2}%
\Big) \psi_j\Big(\frac{x}{|x|}\Big),
\end{equation}
$\psi_j$ is an eigenfunction of the operator $L_{{\mathbf{A}},a}$ on the
sphere $\mathbb{S}^{N-1}$ associated to the $j$-th eigenvalue $\mu_{j}({%
\mathbf{A}},a)$ as in \eqref{angular},  and $P_{j,n}$ is the polynomial of
degree $n$ given by
\begin{equation*}
P_{j,n}(t)=\sum_{i=0}^n \frac{(-n)_i}{\big(\frac{N}2-\alpha_j\big)_i}\,\frac{%
t^i}{i!},
\end{equation*}
denoting as $(s)_i$, for all $s\in{\mathbb{R}}$, the Pochhammer's symbol $%
(s)_i=\prod_{j=0}^{i-1}(s+j)$, $(s)_0=1$.
\end{Proposition}

\proof Assume that $\gamma$ is an eigenvalue of $T_{{\mathbf{A}},a}$ and $%
g\in{\mathcal{H}}\setminus\{0\}$ is a corresponding eigenfunction, so that
\begin{equation}  \label{calculo}
\left(-i\,\nabla+\frac{{\mathbf{A}}\big(\frac{x}{|x|}\big)} {|x|}%
\right)^{\!\!2}g(x) +\dfrac{a\big(\frac{x}{|x|}\big)}{|x|^2}g(x) + \frac{%
|x|^2}{4}\, g(x)=\gamma\, g(x)
\end{equation}
in a weak ${\mathcal{H}}$-sense. From classical elliptic regularity theory,
$g\in C^{1,\alpha}_{\mathrm{loc}}({\mathbb{R}}^N\setminus\{0\},{\mathbb{C}})$%
.  Hence $g$ can be expanded as
\begin{equation*}
g(x)=g(r\theta)=\sum_{k=1}^\infty\phi_k(r)\psi_k(\theta) \quad \text{in }L^2(%
{\mathbb{S}}^{N-1}),
\end{equation*}
where $r=|x|\in(0,+\infty)$, $\theta=x/|x|\in{{\mathbb{S}}^{N-1}}$, and
\begin{equation*}
\phi_k(r)=\int_{{\mathbb{S}}^{N-1}}g(r\theta) \overline{\psi_k(\theta)}%
\,dS(\theta).
\end{equation*}
Equations \eqref{angular} and \eqref{calculo} imply that, for every $k$,
\begin{equation}  \label{ODE}
\phi^{\prime \prime }_{k}+\dfrac{N-1}{r}\phi^{\prime }_{k} +\left(\gamma-%
\dfrac{\mu_k}{r^2}-\dfrac{r^2}{4}\right)\phi_{k}=0 \quad\text{in }%
(0,+\infty).
\end{equation}
Since $g\in {\mathcal{H}}$, we have that
\begin{align}\label{regularidadL2}
\infty>\int_{{\mathbb{R}}^N}g^2(x)\,dx&=\int_{0}^{\infty} \!\bigg(\int_{{%
\mathbb{S}}^{N-1}}g^{2}(r\theta)\,dS(\theta)\bigg) r^{N-1}\,dr\\
&\notag\geq
\int_{0}^{\infty}r^{N-1}\phi_{k}^{2}(r)\,dr
\end{align}
and
\begin{equation}  \label{regularidadHardy}
\infty>\int_{{\mathbb{R}}^N}\dfrac{g^2(x)}{|x|^2}\,dx\geq\int_{0}^{%
\infty}r^{N-3}\phi_{k}^2(r)\,dr.
\end{equation}
For all $k=1,2,\dots$ and $t>0$, we define $w_{k}(t)=(2t)^{\frac{\alpha_k}{2}%
} e^{\frac{t}{2}}\phi_k(\sqrt{2t})$, with $\alpha_{k}$ as in (\ref%
{eigenvalues}). From \eqref{ODE}, $w_k$ satisfies
\begin{equation*}
t w_{k}^{\prime \prime }(t)+\left(\frac{N}{2}-\alpha_k-t\right)w^{\prime
}_{k}(t)- \left(\frac N 4-\frac{\alpha_k}{2}-\frac \gamma
2\right)w_{k}(t)=0\quad\text{in }(0,+\infty).
\end{equation*}
Therefore, $w_{k}$ is a solution of the well known Kummer Confluent
Hypergeometric Equation (see \cite{AS} and \cite{MA}). Then there exist $%
A_k,B_k\in{\mathbb{R}}$ such that
\begin{equation*}
w_k(t)=A_k M\Big(\frac N 4-\frac{\alpha_k}{2}-\frac \gamma 2,\frac
N2-\alpha_k,t\Big)
+B_k U\Big(\frac N 4-\frac{\alpha_k}{2}-\frac \gamma 2,\frac N2-\alpha_k,t%
\Big), \quad t\in (0,+\infty).
\end{equation*}
Here $M(c,b,t)$ and, respectively, $U(c,b,t)$ denote the Kummer function (or
confluent hypergeometric function) and, respectively, the Tricomi function
(or confluent hypergeometric function of the second kind); $M(c,b,t)$ and $%
U(c,b,t)$ are two linearly independent solutions to the Kummer Confluent
Hypergeometric Equation
\begin{equation*}
tw^{\prime \prime }(t)+(b-t)w^{\prime }(t)-cw(t)=0,\quad t\in (0,+\infty).
\end{equation*}
Since $\big(\frac N2-\alpha_k\big)>1$, from the well-known asymptotics of $U$
at $0$ (see e.g. \cite{AS}), we have that
\begin{equation*}
U\Big(\frac N 4-\frac{\alpha_k}{2}-\frac \gamma 2,\frac N2-\alpha_k,t\Big)
\sim \text{\textrm{const}}\,t^{1-\frac{N}{2}+\alpha_k} \quad\text{as }t\to
0^+,
\end{equation*}
for some $\text{\textrm{const}}\neq 0$ depending only on $N,\gamma$, and $%
\alpha_k$. On the other hand, $M$ is the sum of the series
\begin{equation*}
M(c,b,t)=\sum_{n=0}^\infty \frac{(c)_n}{(b)_n}\,\frac{t^n}{n!}.
\end{equation*}
We notice that $M$ has a finite limit at $0^+$, while its behavior at $\infty
$ is singular and depends on the value $-c=-\frac N 4+\frac{\alpha_k}{2}%
+\frac \gamma 2$. If $-\frac N 4+\frac{\alpha_k}{2}+\frac \gamma 2=m\in {%
\mathbb{N}}=\{0,1,2,\cdots\}$, then $M\big(\frac N 4-\frac{\alpha_k}{2}%
-\frac \gamma 2,\frac N2-\alpha_k,t\big)$ is a polynomial of degree $m$ in $t
$, which we will denote as $P_{k,m}$, i.e.,
\begin{equation*}
P_{k,m}(t)=M\Big(-m,{\textstyle{\frac N2}}-\alpha_k,t\Big)= \sum_{n=0}^m
\frac{(-m)_n}{\big(\frac{N}2-\alpha_k\big)_n}\,\frac{t^n}{n!}.
\end{equation*}
If $\big(-\frac N 4+\frac{\alpha_k}{2}+\frac \gamma 2\big)\not\in {\mathbb{N}%
}$, then from the well-known asymptotics of $M$ at $\infty$ (see e.g. \cite%
{AS}) we have that
\begin{equation*}
M\Big(\frac N 4-\frac{\alpha_k}{2}-\frac \gamma 2,\frac N2-\alpha_k,t\Big)
\sim \text{\textrm{const}}\,e^tt^{-\frac{N}{4}+\frac{\alpha_k}2-\frac{\gamma%
}{2}} \quad\text{as }t\to +\infty,
\end{equation*}
for some $\text{\textrm{const}}\neq 0$ depending only on $N,\gamma$, and $%
\alpha_k$.

Now, let us fix $k\in{\mathbb{N}}$, $k\geq 1$. From the above description,
we have that
\begin{equation*}
w_k(t)\sim \mathrm{const\,}B_k t^{1-\frac{N}{2}+\alpha_k} \quad\text{as }%
t\to 0^+,
\end{equation*}
for some $\text{\textrm{const}}\neq 0$, and hence
\begin{equation*}
\phi_k(r)=r^{-\alpha_k}e^{-\frac{r^2}{4}}w_k\Big(\frac{r^2}{2}\Big)\sim
\mathrm{const\,}B_k r^{-(N-2)+\alpha_k} \quad\text{as }r\to 0^+,
\end{equation*}
for some $\text{\textrm{const}}\neq 0$. Therefore, condition (\ref%
{regularidadHardy}) can be satisfied only for $B_k=0$. If $\frac{\alpha_k}{2}%
+\frac{\gamma}{2}-\frac{N}{4}\not\in {\mathbb{N}}$, then
\begin{equation*}
w_k(t)\sim \mathrm{const\,}A_ke^t t^{-\frac{N}{4}+\frac{\alpha_k}2-\frac{%
\gamma}{2}} \quad\text{as }t\to +\infty,
\end{equation*}
for some $\text{\textrm{const}}\neq 0$, and hence
\begin{equation*}
\phi_k(r)=r^{-\alpha_k}e^{-\frac{r^2}{4}}w_k\Big(\frac{r^2}{4}\Big)\sim
\mathrm{const\,}A_k r^{-\frac{N}{2}-\gamma}e^{r^2/4} \quad\text{as }r\to
+\infty,
\end{equation*}
for some $\text{\textrm{const}}\neq 0$. Therefore, condition (\ref%
{regularidadL2}) can be satisfied only for $A_k=0$.  If $\frac{\alpha_k}{2}+%
\frac{\gamma}{2}-\frac{N}{4}=m\in {\mathbb{N}}$, then $r^{-\alpha_k}e^{-%
\frac{r^2}{4}}P_{k,m}\big(\frac{r^2}{2}\big)$ solves (\ref{ODE}); moreover
the function
\begin{equation*}
V_{m,k}(x)=|x|^{-\alpha_k}e^{-\frac{|x|^2}{4}}P_{k,m}\Big(\frac{|x|^2}{2}%
\Big)
\psi_k\Big(\frac{x}{|x|}\Big)
\end{equation*}
belongs to ${\mathcal{H}}$, thus providing an eigenfunction of $L$. \qed

\begin{remark}
\label{rem:harmonic}  If $a(\theta)\equiv0$, ${\mathbf{A}}\equiv{\mathbf{0}}$%
, the spectrum of  $L_{{\mathbf{0}},0}$ is described in Remark \ref%
{rem:free_case}, so  that the spectrum of $T_{{\mathbf{0}},0}$ is $\frac N2+{%
\mathbb{N}}$. Hence, in  this case we recover the eigenvalues of the
Harmonic oscillator  operator $-\Delta+\frac{|x|^2}4$ (see e.g. \cite{KW}).
\end{remark}

\begin{remark}
\label{rem:ortho}  It is easy to verify that
\begin{equation*}
\text{if }(m_1,k_1)\neq(m_2,k_2)\quad\text{then}\quad V_{m_1,k_1}\text{ and }
V_{m_2,k_2}\text{ are orthogonal in }{\  L^{2}({\mathbb{R}}^N)}.
\end{equation*}
By Lemma \ref{Hilbert}, it follows that
\begin{equation*}
\left\{ \widetilde V_{n,j}= \frac{V_{n,j}}{\|V_{n,j}\|_{L^{2}({\mathbb{R}}%
^N)}}: j,n\in{\mathbb{N}},j\geq  1\right\}
\end{equation*}
is an orthonormal basis of $L^{2}({\mathbb{R}}^N)$.
\end{remark}

\begin{remark}
\label{rem:Laguerre} Denoting by $L_{m}^{\alpha}(t)$ the generalized
Laguerre polynomials
\begin{equation*}
L_{m}^{\alpha}(t)=\sum_{n=0}^m (-1)^{n}{\binom{m+\alpha}{m-n}\,\frac{t^n}{n!}%
},
\end{equation*}
and $\beta_{k}=\sqrt{\big(\frac{N-2}{2}\big)^{2}+\mu_k({\mathbf{A}},a)}$ so
that $\gamma_{m,k}=2m+\beta_k+1$, we can write
\begin{equation*}
P_{k,m}\bigg(\frac{|x|^2}{2}\bigg)= M\Big(-\frac{\gamma_{m,k}}{2}+\frac{%
\beta_k}{2}+\frac 12,1+\beta_{k},\frac{|x|^2}{2}\Big)=\binom{m+\beta_k}{m}%
^{\!\!-1} L_{m}^{\beta_k}\Big(\frac{|x|^2}{2}\Big).
\end{equation*}
From the well known orthogonality relation
\begin{equation*}
\int_{0}^{\infty} x^{\alpha} e^{-x} L_{n}^{\alpha}(x) L_{m}^{\alpha}(X)\,dx=%
\frac{\Gamma(n+\alpha+1)}{n!}\delta_{n,m},
\end{equation*}
where $\delta_{n,m}$ denotes the Kronecker delta, it is easy to check that
\begin{equation*}
\|V_{m,k}\|_{L^{2}({\mathbb{R}}^N)}^{2}= 2^{\beta_k}\Gamma(1+\beta_k)\binom{%
m+\beta_k}{m}^{\!\!-1}.
\end{equation*}
\end{remark}

\section{Proof of Theorem
  \protect\ref{Main}}\label{sec:main-theor-repr}

The proof of Theorem \ref{Main} uses Lemma \ref{l:queue}, which is
proved below.

\smallskip\noindent
\begin{pfn}{Lemma \ref{l:queue}}
From Theorem \ref{t:weyl} and Lemma \ref{l:stiautf} in the Appendix,
we deduce that there exist some $k_0\in\N$,
and $C_{i}>0$, $i=1,2$, such that for every $k>k_0$,
\begin{equation}\label{cond1}
-\alpha _{k}>C_{1}k^{\delta _{1}},
\end{equation}
and
\begin{equation}\label{cond2}
\vert \psi _{k}(\theta)\vert <C_{2}k^{\delta
_{2}}\quad\text{for all }\theta\in\SN,
\end{equation}
with $\delta_1=\frac1{N-1}$ and $\delta_2=\frac{2}{N-1}\lfloor
\frac{N-1}2\rfloor$.
From \eqref{cond1} and \eqref{cond2} it follows that, for all $k>k_0$, the $k$-th term of the
series  $K_{k}(x,y)=i^{-\beta _{k}}j_{-\alpha
_{k}}(|x||y|)\psi _{k}\big(\tfrac{x}{|x|}\big)\overline{\psi _{k}\big(\tfrac{%
y}{|y|}\big)}$ belongs to  $L_{\rm loc}^{\infty
}(\R^{2N},{\mathbb{C}})$. Furthermore, if we fix some compact set $\mathcal
K\Subset\R^N$, there exists $k_0'$ such that
$\big|\frac{e|x||y|}{2(-\alpha _{k}+%
\frac{N}{2})}\big|\leq \frac12$ for every $x,y\in\mathcal K$ and
$k> k_0'$. Therefore, for all  $x,y\in\mathcal K$ and
$k> k_0'$, we have that
\begin{eqnarray*}
\left\vert K_{k}(x,y)\right\vert  &\leq &\left\vert j_{-\alpha
_{k}}(|x||y|)\right\vert \left\vert \psi _{k}\big(\tfrac{x}{|x|}\big)%
\right\vert \left\vert \psi _{k}\big(\tfrac{y}{|y|}\big)\right\vert  \\
&\leq &Ck^{2\delta _{2}}\bigg(\frac{|x||y|}{2}\bigg)^{\!\!-\alpha
_{k}}\sum\limits_{m=0}^{\infty }\dfrac{1}{\Gamma (m+1)\Gamma (m-\alpha _{k}+%
\frac{N-2}{2}+1)}\bigg(\frac{|x||y|}{2}\bigg)^{\!\!2m} \\
&\leq &Ck^{2\delta _{2}}\frac{\big(\frac{|x||y|}{2}\big)^{\!\!-\alpha _{k}}%
}{\Gamma (-\alpha _{k}+\frac{N}{2})}e^{\big(\frac{|x||y|}{2}\big)%
^{\!\!2}}\leq C^{\prime }k^{2\delta _{2}}\bigg(\frac{e|x||y|}{2(-\alpha _{k}+%
\frac{N}{2})}\bigg)^{\!\!-\alpha _{k}}e^{\big(\frac{|x||y|}{2}\big)%
^{\!\!2}} \\
&\leq &C^{\prime \prime }k^{2\delta _{2}}\bigg(\frac{1}{2}\bigg)%
^{\!\!C_{1}k^{\delta _{1}}}\equiv M_{k}
\end{eqnarray*}%
where $C^{\prime \prime }$ depends on $\mathcal K$ but not on $k$. Weierstrass
M-test and convergence of $\sum_{k}M_{k}$ yields then the desired uniform
convergence.
\end{pfn}

We are now ready to prove our main result, the representation formula
given by Theorem \ref{Main}.

\begin{pfn}{Theorem \ref{Main}}
Let us expand the initial datum $u_{0}=u(\cdot
,0)=\varphi (\cdot ,0)$  in Fourier series with
respect to the orthonormal basis of $L^{2}({\mathbb{R}}^{N})$
introduced in Remark \ref{rem:ortho} as
\begin{equation}
u_{0}=\sum\limits_{\substack{ m,k\in {\mathbb{N}} \\ k\geq 1}}c_{m,k}%
\widetilde{V}_{m,k}\quad \text{in }L^{2}({\mathbb{R}}^{N}),\quad \text{where
}c_{m,k}=\int_{{\mathbb{R}}^{N}}u_{0}(x)\overline{\widetilde{V}_{m,k}(x)}%
\,dx,  \label{datoinicial}
\end{equation}%
and, for $t>0$, the function $\varphi (\cdot ,t)$ defined
in \eqref{varphi} as
\begin{equation}
\varphi (\cdot ,t)=\sum\limits_{\substack{ m,k\in {\mathbb{N}} \\ k\geq 1}}%
\varphi _{m,k}(t)\widetilde{V}_{m,k}\quad \text{in }L^{2}({\mathbb{R}}^{N}),
\label{eq:exp_varphi}
\end{equation}%
where
\begin{equation*}
\varphi _{m,k}(t)=\int_{{\mathbb{R}}^{N}}\varphi (x,t)\overline{\widetilde{V}%
_{m,k}(x)}\,dx.
\end{equation*}%
Since $\varphi (z,t)$ satisfies \eqref{varphieq}, we obtain that $\varphi
_{m,k}\in C^{1}({\mathbb{R}},{\mathbb{C}})$ and
\begin{equation*}
i\varphi _{m,k}^{\prime }(t)=\dfrac{\gamma _{m,k}}{1+t^{2}}\varphi
_{m,k}(t),\quad \varphi _{m,k}(0)=c_{m,k},
\end{equation*}%
which by integration yields $\varphi _{m,k}(t)=c_{m,k}e^{-i\gamma
_{m,k}\arctan t}$. Hence expansion (\ref{eq:exp_varphi}) can be rewritten as
\begin{equation*}
\varphi (z,t)=\sum\limits_{\substack{ m,k\in {\mathbb{N}} \\ k\geq 1}}%
c_{m,k}e^{-i\gamma _{m,k}\arctan t}\widetilde{V}_{m,k}(z)\quad \text{in }%
L^{2}({\mathbb{R}}^{N}),\quad \text{ for all }t>0.
\end{equation*}%
In view of \eqref{datoinicial}, the above series can be written as
\begin{equation*}
\varphi (z,t)=\sum\limits_{\substack{ m,k\in {\mathbb{N}} \\ k\geq 1}}%
e^{-i\gamma _{m,k}\arctan t}\bigg(\int_{{\mathbb{R}}^{N}}u_{0}(y)\overline{%
\widetilde{V}_{m,k}(y)}\,dy\bigg)\widetilde{V}_{m,k}(z),
\end{equation*}%
in the sense that, for all $t>0$, the above series converges
in $L^{2}({\mathbb{R}}^{N})$. Since $u_{0}(y)$ can be expanded~as
\begin{equation*}
u_{0}(y)=u_{0}\big(|y|\,\tfrac{y}{|y|}\big)=\sum_{j=1}^{\infty
}u_{0,j}(|y|)\psi _{j}\big(\tfrac{y}{|y|}\big)\quad \text{in }L^{2}({\mathbb{%
S}}^{N-1}),
\end{equation*}%
where $u_{0,j}(|y|)=\int_{{\mathbb{S}}^{N-1}}u_{0}(|y|\theta )\overline{\psi
_{j}(\theta )}\,dS(\theta )$, we conclude that
\begin{align*}
\varphi (z,t)& =\sum\limits_{\substack{ m,k\in {\mathbb{N}} \\ k\geq 1}}%
\frac{e^{-i\gamma _{m,k}\arctan t}}{\Vert V_{m,k}\Vert _{L^{2}}^{2}}%
V_{m,k}(z)\bigg(\int_{0}^{\infty }\!\!u_{0,k}(r)r^{N-1-\alpha _{k}}P_{k,m}(%
\tfrac{r^{2}}{2})e^{-\frac{r^{2}}{4}}\,dr\bigg) \\
& =\sum\limits_{k=1}^{\infty }\psi _{k}\big(\tfrac{z}{|z|}\big)\frac{%
e^{-i(\beta _{k}+1)\arctan t}}{2^{\beta _{k}}\Gamma (1+\beta _{k})}\!\left[
\sum\limits_{m=0}^{\infty }\frac{{\textstyle{\binom{m+\beta _{k}}{m}}}}{%
e^{i2m\arctan t}}\times \right.  \\
& \ \ \ \left. \times \bigg(\int_{0}^{\infty }\!\frac{u_{0,k}(r)}{%
|rz|^{\alpha _{k}}}P_{k,m}\big(\tfrac{r^{2}}{2}\big)P_{k,m}\big(\tfrac{%
|z|^{2}}{2}\big)e^{-\frac{r^{2}+|z|^{2}}{4}}r^{N-1}dr\!\bigg)\!\right] \!.
\end{align*}%
By \cite{AS} we know that
\begin{equation*}
P_{k,m}\Big(\frac{r^{2}}{2}\Big)=\dfrac{\Gamma (1+\beta _{k})}{\Gamma
(1+\beta _{k}+m)}e^{\frac{r^{2}}{2}}r^{-\beta _{k}}2^{\frac{\beta _{k}}{2}}%
\displaystyle\int_{0}^{\infty }e^{-t}t^{m+\frac{\beta _{k}}{2}}J_{\beta
_{k}}(\sqrt{2}r\sqrt{t})\,dt,
\end{equation*}%
where $J_{\beta _{k}}$ is the Bessel function of the first kind of order $%
\beta _{k}$. Therefore,
\begin{align*}
\varphi (z,t)& =4\sum\limits_{k=1}^{\infty }\psi _{k}\big(\tfrac{z}{|z|}\big)%
\frac{\Gamma (1+\beta _{k})}{e^{i(\beta _{k}+1)\arctan t}}\Bigg[%
\sum\limits_{m=0}^{\infty }\binom{m+\beta _{k}}{m}\frac{e^{-i2m\arctan t}}{%
(\Gamma (1+\beta _{k}+m))^{2}}\times  \\
& \ \ \ \times \!\bigg(\!\int_{0}^{\infty }\frac{u_{0,k}(r)}{|rz|^{\alpha
_{k}+\beta _{k}}}e^{\frac{r^{2}+|z|^{2}}{4}}\left( \int_{0}^{\infty
}\!\!\!\int_{0}^{\infty }\!\!e^{-s^{2}-{s^{\prime }}^{2}}\!(ss^{\prime})^{
2m+\beta _{k}+1}\times \right.  \\
& \ \ \ \left. \times J_{\beta _{k}}(\sqrt{2}rs)J_{\beta _{k}}(\sqrt{2}%
|z|s^{\prime })\,ds\,ds^{\prime }\right) \!r^{N-1}dr\!\bigg)\Bigg] \\
& =4\sum\limits_{k=1}^{\infty }\psi _{k}\big(\tfrac{z}{|z|}\big)e^{-i(\beta
_{k}+1)\arctan t}\!\Bigg[\int_{0}^{\infty }u_{0,k}(r)|rz|^{-\alpha
_{k}-\beta _{k}}e^{\frac{r^{2}+|z|^{2}}{4}}r^{N-1}\times  \\
& \ \ \ \times e^{i(\arctan t+\frac{\pi }{2})\beta _{k}}\bigg(%
\int_{0}^{\infty }\!\!\!\int_{0}^{\infty }\!\!\!\frac{ss^{\prime }}{%
e^{s^{2}+s^{\prime 2}}}\times  \\
& \ \ \ \times \bigg(\sum\limits_{m=0}^{\infty }\frac{(-1)^{m}e^{-i(\arctan
t+\frac{\pi }{2})(2m+\beta _{k})}}{\Gamma (1+m)\Gamma (1+\beta _{k}+m)}%
(ss^{\prime})^{ 2m+\beta _{k}}\bigg)\times  \\
& \ \ \ \times J_{\beta _{k}}(\sqrt{2}rs)J_{\beta _{k}}(\sqrt{2}|z|s^{\prime
})ds\,ds^{\prime }\!\bigg)dr\Bigg].
\end{align*}%
Then, since $\sum\limits_{m=0}^{\infty }(-1)^{m}\frac{e^{-i(\arctan
    t+\frac{%
      \pi }{2})(2m+\beta _{k})}}{\Gamma (1+m)\Gamma (1+\beta
  _{k}+m)}(ss^{\prime})^{2m+\beta _{k}}=J_{\beta _{k}}(2e^{-i(\arctan
  t+\frac{\pi }{2})}ss^{\prime })$%
, we have
\begin{multline}
\varphi (z,t)  \label{eq:2} \\
=4\sum\limits_{k=1}^{\infty }\psi _{k}\big(\tfrac{z}{|z|}\big)e^{i(\beta _{k}%
\frac{\pi }{2}-\arctan t)}\!\Bigg[\int_{0}^{\infty }u_{0,k}(r)|rz|^{-\frac{%
N-2}{2}}e^{\frac{r^{2}+|z|^{2}}{4}}{\mathcal{I}}_{k,t}(r,|z|)r^{N-1}dr\Bigg],
\end{multline}%
where
\begin{equation*}
{\mathcal{I}}_{k,t}(r,|z|)=\int_{0}^{\infty }\!\!\!\int_{0}^{\infty
}\!\!ss^{\prime} e^{-s^{2}-s^{\prime 2}}J_{\beta _{k}}(2e^{-i(\arctan t+\frac{%
\pi }{2})}ss^{\prime })J_{\beta _{k}}(\sqrt{2}rs)J_{\beta _{k}}(\sqrt{2}%
|z|s^{\prime })\,ds\,ds^{\prime }.
\end{equation*}%
From \cite[formula (1), p. 395]{watson} (with $t=s^{\prime }$, $p=1$, $a=%
\sqrt{2}|z|$, $b=2e^{-i(\arctan t+\frac{\pi }{2})}s$, $\nu =\beta _{k}$
which satisfy $\Re (\nu )>-1$ and $|\mathop{\rm arg}p|<\frac{\pi }{4}$), we
know that
\begin{align*}
& \int_{0}^{\infty }\!\!s^{\prime} e^{-s^{\prime 2}}J_{\beta
_{k}}(2e^{-i(\arctan t+\frac{\pi }{2})}ss^{\prime })J_{\beta _{k}}(\sqrt{2}%
|z|s^{\prime })\,ds^{\prime } \\
& =\frac{1}{2}e^{-\frac{|z|^{2}+2e^{-i(2\arctan t+\pi )}s^{2}}{2}}I_{\beta
_{k}}\bigg(\frac{\sqrt{2}|z|s}{e^{i(\arctan t+\frac{\pi }{2})}}\bigg),
\end{align*}%
where $I_{\beta _{k}}$ denotes the modified Bessel function of order $\beta
_{k}$. Hence
\begin{align*}
& {\mathcal{I}}_{k,t}(r,|z|) \\
& \ \ =\dfrac{1}{2}\displaystyle\int_{0}^{\infty }se^{-s^{2}}J_{\beta _{k}}(%
\sqrt{2}rs)e^{-\frac{|z|^{2}+2e^{-i(2\arctan t+\pi )}s^{2}}{2}}I_{\beta
_{k}}(\sqrt{2}e^{-i(\arctan t+\frac{\pi }{2})}|z|s)\,ds \\
& \ \ =\dfrac{1}{4}\displaystyle\int_{0}^{\infty }se^{-\frac{s^{2}}{2}%
}J_{\beta _{k}}(rs)e^{-\frac{|z|^{2}+e^{-i(2\arctan t+\pi )}s^{2}}{2}%
}I_{\beta _{k}}(e^{-i(\arctan t+\frac{\pi }{2})}|z|s)\,ds.
\end{align*}%
Since $I_{\nu }(x)=e^{-\frac{1}{2}\nu \pi i}J_{\nu }(xe^{\frac{\pi }{2}i})$
(see e.g. \cite[9.6.3, p. 375]{AS}), we obtain
\begin{equation*}
{\mathcal{I}}_{k,t}(r,|z|)=\frac{1}{4}e^{-\frac{\beta _{k}}{2}\pi i}e^{-
\frac{|z|^{2}}{2}}\!\!\!\int_{0}^{\infty }\!\!\!se^{-\frac{s^{2}}{2}
(e^{-i(2\arctan t+\pi )}+1)}J_{\beta _{k}}(rs)J_{\beta _{k}}(e^{-i\arctan
t}|z|s)\,ds.
\end{equation*}%
Applying \cite[formula (1), p. 395]{watson} (with $t=s$, $p^{2}=\frac{{%
1+e^{-i(2\arctan t+\pi )}}}{2}$, $a=r$, $b=e^{-i\arctan t}|z|$, $\nu =\beta
_{k}$ which satisfy $\Re (\nu )>-1$ and $|\mathop{\rm arg}p|<\frac{\pi }{4}$%
) and \cite[9.6.3, p. 375]{AS}, we obtain
\begin{align}
& {\mathcal{I}}_{k,t}(r,|z|)  \label{eq:calI} \\
=\dfrac{1}{4}e^{-\frac{\beta _{k}}{2}\pi i}& e^{-\frac{|z|^{2}}{2}}\dfrac{1}{%
1+e^{-i(2\arctan t+\pi )}}e^{-\frac{r^{2}+|z|^{2}e^{-2i\arctan t}}{%
2(1+e^{-i(2\arctan t+\pi )})}}I_{\beta _{k}}\bigg(\dfrac{r|z|e^{-i\arctan t}%
}{1+e^{-i(2\arctan t+\pi )}}\bigg)  \notag \\
=\frac{1}{4}e^{-\beta _{k}\pi i}& e^{-\frac{|z|^{2}}{2}}\dfrac{1}{%
1+e^{-i(2\arctan t+\pi )}}e^{-\frac{r^{2}+|z|^{2}e^{-2i\arctan t}}{%
2(1+e^{-i(2\arctan t+\pi )})}}J_{\beta _{k}}\bigg(i\dfrac{r|z|e^{-i\arctan t}%
}{1+e^{-i(2\arctan t+\pi )}}\bigg).  \notag
\end{align}%
Noticing that
\begin{equation*}
e^{-i\arctan t}=-\frac{i(t+i)}{\sqrt{1+t^{2}}},\quad \dfrac{1}{%
1+e^{-i(2\arctan t+\pi )}}=\frac{t-i}{2t},
\end{equation*}%
from \eqref{eq:2} and \eqref{eq:calI} we deduce
\begin{multline}
\varphi (z,t)  \label{eq:3} \\
=\frac{e^{-i\arctan t}}{1+e^{-i(2\arctan t+\pi )}}\sum\limits_{k=1}^{\infty
}\psi _{k}\big(\tfrac{z}{|z|}\big)e^{-i\beta _{k}\frac{\pi }{2}}\!\Bigg[%
\int_{0}^{\infty }\frac{u_{0,k}(r)}{|rz|^{\frac{N-2}{2}}}e^{\frac{r^{2}}{2}%
\big(\frac{1}{2}-\frac{1}{1+e^{-i(2\arctan t+\pi )}}\big)}\times  \\
\times e^{-\frac{|z|^{2}}{4}\big(1+\frac{2e^{-2i\arctan t}}{1+e^{-i(2\arctan
t+\pi )}}\big)}J_{\beta _{k}}\bigg(i\dfrac{r|z|e^{-i\arctan t}}{%
1+e^{-i(2\arctan t+\pi )}}\bigg)r^{N-1}\,dr\Bigg] \\
\ =\frac{\sqrt{1+t^{2}}}{2ti}\sum\limits_{k=1}^{\infty }\psi _{k}\big(\tfrac{%
z}{|z|}\big)e^{-i\beta _{k}\frac{\pi }{2}}\!\Bigg[\int_{0}^{\infty }\frac{%
u_{0,k}(r)}{|rz|^{\frac{N-2}{2}}}e^{-\frac{r^{2}}{4it}}e^{-\frac{|z|^{2}}{4it%
}}J_{\beta _{k}}\bigg(\dfrac{r|z|\sqrt{1+t^{2}}}{2t}\bigg)r^{N-1}\,dr\Bigg].
\end{multline}%
From (\ref{eq:3}) and \eqref{varphi} we get that for $t>0$,
\begin{align}
& u(x,t)=(1+t^{2})^{-\frac{N}{4}}e^{\frac{it|x|^{2}}{4(1+t^{2})}}\varphi %
\bigg(\frac{x}{\sqrt{1+t^{2}}},t\bigg)  \label{serie} \\
=e^{-\frac{|x|^{2}}{4ti}}& \dfrac{1}{2ti}|x|^{-\frac{N-2}{2}%
}\sum\limits_{k=1}^{\infty }\psi _{k}\big(\tfrac{x}{|x|}\big)e^{-i\beta _{k}%
\frac{\pi }{2}}\!\bigg(\int_{0}^{\infty }u_{0,k}(r)e^{-\frac{r^{2}}{4ti}%
}J_{\beta _{k}}\bigg(\dfrac{r|x|}{2t}\bigg)r^{\frac{N}{2}}\,dr\bigg).  \notag
\end{align}%
Notice that, by replacing $\int_{0}^{\infty }$ by $\int_{0}^{R}$ in %
\eqref{serie} one obtains the series representation of the solution $%
u_{R}(x,t)$ with initial data $u_{0,R}(x)\equiv \chi _{R}(x)u_{0}(x)$ with $%
\chi _{R}(x)$ the characteristic function of the ball of radius $R$ centered
at the origin. Since the evolution by Schr\"{o}dinger equation is an
isometry in $L^{2}$, we have that for all $t\in\R$ $\left\Vert u-u_{R}\right\Vert
_{L^{2}}(t)=\left\Vert u_{0}-u_{0,R}\right\Vert _{L^{2}}\rightarrow 0$, as $%
R\rightarrow \infty $. Hence $u(\cdot,t)=\lim_{R\rightarrow \infty }u_{R}(\cdot,t)$ in $L^{2}(\R^N)$. Since
\begin{equation*}
u_{0,k}(r)=\int_{{\mathbb{S}}^{N-1}}u_{0}(r\theta )\overline{\psi
_{k}(\theta )}\,dS(\theta ),
\end{equation*}%
and, by hypothesis, the queue of the series
\begin{equation*}
K(x,y)=\sum\limits_{k=1}^{\infty }e^{-i\beta _{k}\frac{\pi }{2}}\psi _{k}%
\big(\tfrac{x}{|x|}\big)\overline{\psi _{k}(\theta )}J_{\beta _{k}}\bigg(%
\dfrac{r|x|}{2t}\bigg)
\end{equation*}%
is uniformly convergent on compacts, we can exchange integral and sum and
write
\begin{align*}
  u_{R}(x,t)&=
  \frac{e^{-\frac{|x|^{2}}{4ti}}}{2ti}|x|^{-\frac{N-2}{2}%
  }\!\int_{0}^{R}\!\!\int_{{\mathbb{S}}^{N-1}}u_{0}(r\theta
  )r^{\frac{N}{2}
  }e^{-\frac{r^{2}}{4ti}}\times  \\
  & \quad\times\bigg[\sum\limits_{k=1}^{\infty }e^{-i\beta
    _{k}\frac{\pi }{2}}\psi _{k}%
  \big(\tfrac{x}{|x|}\big)\overline{\psi _{k}(\theta )}J_{\beta
    _{k}}\bigg(
  \dfrac{r|x|}{2t}\bigg)\bigg]dr\,dS(\theta )\\
  &= \frac{e^{\frac{i|x|^{2}}{4t}}}{i(2t)^{{N}/{2}}}
  \int_{B_R}K\bigg(\frac{x}{\sqrt{2t}},\frac{y}{\sqrt{2t}}\bigg)e^{i\frac{|y|^{2}}{%
      4t}}u_{0}(y)\,dy .
\end{align*}
Letting $R\rightarrow \infty $, we obtain \eqref{representation} thus
completing the proof of Theorem \ref{Main}.
\end{pfn}

\section{Proof of Theorem
  \protect\ref{thm:AB}}\label{sec:ahar-bohm-magn}

\noindent 
In view of Remark \ref{rem:t_neg}, it is enough to prove the stated
estimate for $t>0$. Moreover, thanks to Corollary \ref{cor:decay}, it is sufficient to prove condition
\eqref{eq:claim}, namely,  uniform boundedness of
\begin{equation*}
K(x,y)=\frac{1}{(2\pi)^2}\sum\limits_{j\in {\mathbb{Z}}}
e^{-i|\alpha -j|\frac{\pi }{2}}e^{-ij\arctan \frac{x_{2}}{x_{1}}
}e^{ij\arctan \frac{y_{2}}{y_{1}}}J_{|\alpha -j|}(|x||y|)
\end{equation*}
which can be written as $K(x,y)=\frac{1}{(2\pi)^2}W\big(\arctan
\frac{x_{2}}{x_{1}}-\arctan \frac{y_{2}}{y_{1}},|x||y|\big)$
where
\begin{equation*}
W(z,s)=\sum\limits_{j\in {\mathbb{Z}}}e^{-i|\alpha -j|\frac{\pi }{2}%
}e^{-ijs}J_{|\alpha -j|}\left( z\right).
\end{equation*}
Notice that
\begin{equation*}
\left\vert \alpha -j\right\vert =
\begin{cases}
\alpha -j,&\text{if }j<\alpha, \\
j-\alpha ,& \text{if }j>\left[ \alpha \right] \equiv j_{0},
\end{cases}
\end{equation*}
so that we can write%
\begin{align*}
W(z,s)&=  i^{-\alpha }\sum\limits_{j<-\left\vert j_{0}\right\vert+1
}i^{j}e^{-ijs}J_{\alpha -j}\left( z\right) \\
&\quad +\sum\limits_{-\left\vert j_{0}\right\vert+1 }^{\left\vert j_{0}\right\vert+1
}e^{-i|\alpha -j|\frac{\pi }{2}}e^{-ijs}J_{|\alpha -j|}\left( z\right)
+i^{\alpha }\sum\limits_{j>\left\vert j_{0}\right\vert+1
}i^{-j}e^{-ijs}J_{j-\alpha }\left( z\right)\\
&\equiv i^{-\alpha }S_{1}(z,s)+S_{2}(z,s)+i^{\alpha }S_{3}(z,s).
\end{align*}
$S_{2}(z,s)$ is clearly bounded.
By using identity $9.1.27$ in \cite{AS},
\begin{equation*}
J_{\nu }^{\prime }(z)=\frac{1}{2}(J_{\nu -1}(z)-J_{\nu +1}(z)),
\end{equation*}%
we can compute%
\begin{align}  \label{iden1}
\frac{d}{dz}&S_{1}(z,s)  =\frac{1}{2}\sum\limits_{j>\left\vert j_{0}\right\vert+1
}i^{-j}e^{ijs}(J_{\alpha +j-1}(z)-J_{\alpha +j+1}(z)) \\
&=\frac{1}{2}\sum\limits_{j>\left\vert j_{0}\right\vert}i^{-(j+1)}e^{i(j+1)s}J_{\alpha +j}(z)-\frac{1}{2}\sum\limits_{j>\left%
\vert j_{0}\right\vert +2}i^{-(j-1)}e^{i(j-1)s}J_{\alpha +j}(z)  \notag \\
& =\frac{1}{2}\left( i^{-1}e^{is}J_{\alpha +\left\vert j_{0}\right\vert+1
}(z)+J_{\alpha +\left\vert j_{0}\right\vert +2}(z)\right) i^{-\left\vert
j_{0}\right\vert }e^{i\left\vert j_{0}\right\vert s}-\left( i\cos s\right)
S_{1}(z,s),  \notag
\end{align}%
\begin{align}  \label{iden2}
&\frac{d}{dz} S_{3}(z,s)  =\frac{1}{2}\sum\limits_{j>\left\vert j_{0}\right\vert+1
}i^{-j}e^{-ijs}(J_{-\alpha +j-1}(z)-J_{-\alpha +j+1}(z)) \\
& =\frac{1}{2}\sum\limits_{j>\left\vert j_{0}\right\vert
}i^{-(j+1)}e^{-i(j+1)s}J_{-\alpha +j}(z)-\frac{1}{2}\sum\limits_{j>\left%
\vert j_{0}\right\vert +2}i^{-(j-1)}e^{-i(j-1)s}J_{-\alpha +j}(z)  \notag \\
& =\frac{1}{2}\left( i^{-1}e^{-is}J_{-\alpha +\left\vert j_{0}\right\vert+1
}(z)+J_{-\alpha +\left\vert j_{0}\right\vert +2}(z)\right) i^{-\left\vert
j_{0}\right\vert }e^{-i\left\vert j_{0}\right\vert s}-\left( i\cos s\right)
S_{3}(z,s).  \notag
\end{align}%
Defining%
\begin{equation*}
F(z,s)=i^{-\alpha }S_{1}(z,s)+i^{\alpha }S_{3}(z,s),
\end{equation*}%
we deduce from (\ref{iden1}) and (\ref{iden2}) that, for every $s$,
$F(\cdot,s)$ satisfies the differential equation
\begin{equation}
\frac{d}{dz} F(z,s)+\left( i\cos s\right) F(z,s)=g(z,s) ,  \label{de1}
\end{equation}%
with%
\begin{align*}
g(z,s)=&\frac{i^{-\alpha }}{2}\left( i^{-1}e^{is}J_{\alpha +\left\vert
j_{0}\right\vert+1 }(z)+J_{\alpha +\left\vert j_{0}\right\vert +2}(z)\right)
i^{-\left\vert j_{0}\right\vert -1}e^{i(\left\vert j_{0}\right\vert+1) s} \\
\\
&+\frac{i^{\alpha }}{2}\left( i^{-1}e^{-is}J_{-\alpha +\left\vert
j_{0}\right\vert+1 }(z)+J_{-\alpha +\left\vert j_{0}\right\vert +2}(z)\right)
i^{-\left\vert j_{0}\right\vert -1}e^{-i(\left\vert j_{0}\right\vert+1) s}.%
\end{align*}
By integration of (\ref{de1}) we obtain that
\begin{equation*}
F(z,s)=e^{-iz\cos s}\bigg(F(0,s)+\int_0^{z}e^{iz^{\prime }\cos
s}g(z^{\prime },s)dz^{\prime }\bigg).
\end{equation*}%
Since $\left\vert j_{0}\right\vert+1 \pm \alpha >0$,
by the asymptotic behavior of Bessel functions close to the
origin (see formula $9.1.7$ in \cite{AS})
\begin{equation*}
J_{\nu }(x)\simeq \frac{1}{\Gamma (\nu +1)}\left( \frac{x}{2}\right) ^{\nu }
\end{equation*}%
we
conclude that
$F(0,s)=0$,
and hence
\begin{equation*}
F(z,s)=\int_{0}^{z}e^{-i(z-z^{\prime })\cos s}g(z^{\prime },s)dz^{\prime }.
\end{equation*}%
Uniform (in $s$ and $z$) boundedness of $F(z,s)$ follows from uniform
boundedness of the function%
\begin{equation*}
f(z,s)\equiv \int_{0}^{z}e^{iz^{\prime }\cos s}\left( i^{-1}e^{is}J_{\sigma
}(z^{\prime })+J_{\sigma +1}(z^{\prime })\right) dz^{\prime }
\end{equation*}%
for any $\sigma >0$. In order to prove it, we use the identity%
\begin{equation}\label{eq:asintinfti}
J_{\sigma }(z)=\sqrt{\frac{2}{\pi z}}\cos \left( z-\frac{\sigma \pi }{2}-%
\frac{\pi }{4}\right) +\xi_{\sigma} (z)
\end{equation}
with%
\begin{equation}
\left\vert \xi_{\sigma} (z)\right\vert \leq \frac{C_{\sigma}}{z^{\frac{1}{2}}(1+z)},
\label{ineqs}
\end{equation}%
which is a simple consequence of the asymptotic behavior of Bessel
functions at infinity (see formula $9.2.1$ in \cite{AS}).
Therefore,
\begin{align*}
f(z,s) =&\int_{0}^{z}\sqrt{\frac{2}{\pi z^{\prime }}}e^{iz^{\prime }\cos
s}\left( e^{-i\frac{\pi }{2}}e^{is}\cos \left( z^{\prime }-\frac{\sigma \pi
}{2}-\frac{\pi }{4}\right) +\sin \left( z^{\prime }-\frac{\sigma \pi }{2}-%
\frac{\pi }{4}\right) \right) dz^{\prime } \\
&+\int_{0}^{z}e^{iz^{\prime }\cos s}(i^{-1}e^{is}\xi_\sigma (z^{\prime })+\xi_{\sigma+1} (z^{\prime }))dz^{\prime } \equiv
I_{1}(z)+I_{2}(z).
\end{align*}
By (\ref{ineqs}), $I_{2}(z)$ is uniformly bounded.
We notice now that
\begin{gather*}
\sqrt{\frac{2}{\pi z'}}e^{iz^{\prime }\cos s}\left( e^{-i\frac{\pi }{2}%
}e^{is}\cos \left( z'-\frac{\sigma \pi }{2}-\frac{\pi }{4}\right) +\sin
\left( z'-\frac{\sigma \pi }{2}-\frac{\pi }{4}\right) \right)\\
=\frac{1}{2i}\sqrt{\frac{2}{\pi z'}}e^{iz^{\prime }\cos s}\left( \left(
e^{is}+1\right) e^{i\left( z'-\frac{\sigma \pi }{2}-\frac{\pi }{4}\right)
}+\left( e^{is}-1\right) e^{-i\left( z'-\frac{\sigma \pi }{2}-\frac{\pi }{4}%
\right) }\right)
\end{gather*}
and since%
\begin{eqnarray*}
\int_{0}^{z}\frac{e^{iz^{\prime }(\cos s+1)}}{\sqrt{z^{\prime }}}dz^{\prime
} &=&\frac{1}{\sqrt{\cos s+1}}\int_{0}^{z(\cos s+1)}\frac{e^{iy}}{\sqrt{y}}dy
\\
\int_{0}^{z}\frac{e^{iz^{\prime }(\cos s-1)}}{\sqrt{z^{\prime }}}dz^{\prime
} &=&\frac{1}{\sqrt{1-\cos s}}\int_{0}^{z(1-\cos s)}\frac{e^{-iy}}{\sqrt{y}}%
dy
\end{eqnarray*}%
and%
\begin{equation*}
\left\vert \frac{\left( e^{is}+1\right) }{\sqrt{\cos s+1}}\right\vert \leq
C,\quad \left\vert \frac{\left( e^{is}-1\right) }{\sqrt{1-\cos s}}\right\vert
\leq C,
\end{equation*}%
we conclude that $I_{1}(z)$ is uniformly bounded.

Therefore, $K(x,y)$ is uniformly
bounded and then inequality \eqref{eq:decayAB} follows by  Corollary \ref{cor:decay}. \qed

\section{Proof of Theorem \protect\ref{thm:inversesquare}}\label{sec:inverse}

In view of Remark \ref{rem:t_neg}, it is sufficient to consider the
case $t>0$. Let $N=3$, $a>-\frac14$, and $K$ as in
\eqref{eq:Sinverse}. The proof of the theorem will follow from the
following estimates for $K$:
\begin{align}
\label{eq:estimate_i_K}
&\text{if }\alpha_1<0,\quad\text{then}\quad\sup_{x,y\in\R^3}|K(x,y)|<+\infty,\\
\label{eq:estimate_ii_K}
&\text{if
}\alpha_1>0,\quad\text{then}\quad\sup_{x,y\in\R^3}\frac{|K(x,y)|}{1+(|x||y|)^{-\alpha_1}}<+\infty.
\end{align}
Before proving the above estimates, let us show how
\eqref{eq:estimate_i_K} and \eqref{eq:estimate_ii_K} imply estimates
\eqref{eq:decayinverse} and \eqref{eq:decayinverse2} respectively, thus
proving Theorem \ref{thm:inversesquare}.

We notice that if
$\alpha_1=0$ then $a=0$ and there is nothing to prove since in this
case the result reduces to classical decay estimates for the free
Schr\"{o}dinger equation.

If $\alpha_1<0$, in view of
\eqref{eq:estimate_i_K}
 estimate \eqref{eq:decayinverse} directly follows from Corollary \ref{cor:decay}.

If $\alpha _{1}>0$, then \eqref{eq:estimate_ii_K} and
\eqref{representation} imply that, for some  $C>0$ (independent of $x$ and $t$),
\begin{align*}
\left\vert u(x,t)\right\vert &\leq
\frac{C}{t^{\frac32}}\int_{\R^3}\bigg(1+\frac{|x|^{-\alpha_1}|y|^{-\alpha_1}}{t^{-\alpha_1}} \bigg)|u_0(y)|\,dy\\
&=\frac{C}{t^{\frac{3}{2}}}\left\Vert
u_{0}\right\Vert _{L^{1}(\R^3)}+\frac{C}{t^{\frac{3}{2}-\alpha _{1}}}\frac{1}{
\left\vert x\right\vert ^{\alpha _{1}}}\int_{\mathbb{R}^3} \frac{\left\vert
u_{0}(y)\right\vert }{\left\vert y\right\vert ^{\alpha _{1}}}dy,\
\end{align*}
for a.e. $x\in\R^3$ and all $t\geq0$,
 which implies%
\begin{equation}\label{eq:est-w}
\frac{\left\vert x\right\vert ^{\alpha _{1}}}{1+\left\vert x\right\vert
^{\alpha _{1}}}\left\vert u(x,t)\right\vert \leq C\frac{1+t^{\alpha _{1}}}{
t^{\frac{3}{2}}}\int_{\mathbb{R}} \frac{1+\left\vert y\right\vert ^{\alpha
_{1}}}{\left\vert y\right\vert ^{\alpha _{1}}} \left\vert
u_{0}(y)\right\vert dy.
\end{equation}
Let us introduce the weight function $w(y)=\big( \frac{1+| y|^{\alpha _{1}}}{
  |y|^{\alpha _{1}}}\big)^{2}$ and the weighted $L^{p}$ norm
\begin{equation*}
\left\Vert v\right\Vert _{L_{w}^{p}}\equiv
\begin{cases}
 \left( \int_{\R^3} |v(y)|^{p}w(y)dy\right)^{{1}/{p}
},&\text{if }1\leq p<+\infty,\\
\mathop{\rm ess\,sup}_{y\in\R^3}|v(y)|,&\text{if }p=+\infty.
\end{cases}
\end{equation*}
$L^2$conservation and \eqref{eq:est-w} yield the estimates
\begin{equation*}
  \left\Vert \frac{u(\cdot,t)}{\sqrt{w}}\right\Vert _{L_{w}^{2}}
  =\left\Vert \frac{u_{0}}{\sqrt w}\right\Vert
  _{L_{w}^{2}} ,\quad
  \left\Vert \frac{u(\cdot,t)}{\sqrt w}\right\Vert _{L_{w}^{\infty }}
  \leq C\,
  \frac{1+t^{\alpha _{1}}}{t^{\frac{3}{2}}}\left\Vert
    \frac{u_{0}}{\sqrt w}\right\Vert _{L_{w}^{1}}.
\end{equation*}
Then, letting, for all $p>2$,
$$
\theta_p=1-\frac 2p,\quad p'=\frac{p}{p-1},
$$
so that
$$
\theta_p\in(0,1),\quad \frac1{p'}=\frac{1-\theta_p}2+\frac\theta1,\quad
\frac1{p}=\frac{1-\theta_p}2+\frac\theta\infty,
$$
the Riesz-Thorin interpolation theorem yields
\begin{equation*}
  \bigg\|\frac{u(\cdot,t)}{\sqrt w}\bigg\|_{L_w^p}\leq
  C^{\theta_p}\bigg(
\frac{1+t^{\alpha _{1}}}{t^{\frac{3}{2}}}\bigg)^{\theta_p}  \bigg\|\frac{u_0}{\sqrt w}\bigg\|_{L_w^{p'}}
\end{equation*}
i.e.
\begin{multline*}
  \bigg(\int_{\R^3}|u(y,t)|^p \big( \tfrac{1+| y|^{\alpha _{1}}}{
    |y|^{\alpha _{1}}}\big)^{2-p}dy\bigg)^{\!\!1/p}\\
  \leq C^{1-\frac 2p}\bigg( \frac{1+t^{\alpha
      _{1}}}{t^{\frac{3}{2}}}\bigg)^{1-\frac 2p}
  \bigg(\int_{\R^3}|u_0(y)|^{p'} \big( \tfrac{1+| y|^{\alpha _{1}}}{
    |y|^{\alpha _{1}}}\big)^{2-p'}dy\bigg)^{\!\!1/p'}.
\end{multline*}
Hence, inequality \eqref{eq:decayinverse2} in Theorem \ref%
{thm:inversesquare} follows.
Therefore, in order to prove the theorem, it is sufficient to prove
estimates \eqref{eq:estimate_i_K} and \eqref{eq:estimate_ii_K}.

It is well known that the link between plane waves and a
combination of zonal functions is given by the Jacobi-Anger expansion,
combined with the addition theorem for spherical harmonics (see for example
\cite{watson}, \cite{MU} and the references therein). For $N=3$, we get that
\begin{equation}  \label{anger}
e^{-i x\cdot y}=4\pi \sqrt{\dfrac{\pi}{2}} \sum_{\ell=0}^\infty i^{-\ell}
j_{\ell}(|x||y|)Z_{x/|x|}^{(\ell)}(y/|y|).
\end{equation}
We need to estimate the kernel $K$ in \eqref{nucleo} which, as
observed in \eqref{eq:Sinverse}, can be written as
$$
K(x,y)=S\Big(|x||y|,\tfrac{x}{|x|},\tfrac{y}{|y|}\Big)
$$
where
\begin{equation*}
S(r,\theta,\theta')
=\sum_{\ell=0}^{\infty
  }i^{-b_\ell}\,j_{-a_\ell} (r)Z_{\theta}^{(\ell)}(\theta'),
\end{equation*}
with $b_\ell=\sqrt{(\ell+1/2)^2+a}$,
$a_\ell=\frac12-\sqrt{(\ell+1/2)^2+a}$. We split the sum into two terms
\begin{align}\label{suma}
  S(r,\theta ,\theta^{\prime })&=\sum_{\ell=0}^{\ell_{0}-1}i^{-b_{\ell}}
  j_{-a_{\ell}}(r)Z_{\theta }^{(\ell)}\left( \theta^{\prime
    }\right)+\sum_{\ell=\ell_{0}}^{\infty }i^{-b_{\ell}}j_{-a_{\ell}}(r)Z_{\theta
  }^{(\ell)}\left( \theta^{\prime }\right)\\
\notag&=S_{1}(r,\theta ,\theta^{\prime })+S_{2}(r,\theta ,\theta^{\prime }),
\end{align}
with $\ell_{0}\geq0$ such that $a_{\ell}>0$ for any $\ell<\ell_{0}$
and $a_{\ell}<0$ for any $\ell\geq \ell_{0}$
($S_1$ is meant to be zero if $\ell_0=0$).
Our goal is to show that the singularities in the Jacobi-Anger expansion are
described by the first finite sum $S_{1}$ at the right-hand side of (\ref%
{suma}) while the second term $S_{2}$ at the right-hand side is uniformly
bounded. Such boundedness for $S_{2}$ follows from the arguments below. We
have that
\begin{align}\label{suma2}
  S_{2}&=\sum\limits_{\ell=\ell_{0}}^{\infty
  }i^{-b_{l}}j_{-a_{l}}(r)Z_{\theta }^{(\ell)}
  \left( \theta^{\prime }\right) \\
  \notag&=\sum\limits_{\ell=\ell_{0}}^{\infty }i^{-(\ell+\frac{1}{2})}
  j_{\ell}(r) Z_{\theta }^{(\ell)}\left( \theta^{\prime
    }\right)+\sum\limits_{\ell=\ell_{0}}^{\infty }
  (i^{-b_{\ell}}j_{-a_{\ell}}(r)-i^{-(\ell+\frac{1}{2})}j_{\ell}(r))Z_{\theta
  }^{(\ell)}\left( \theta^{\prime }\right)\\
  \notag&=i^{-\frac{1}{2}}\left[(2\pi)^{-\frac32}
    e^{-ir\theta\cdot\theta'}-
    \sum_{\ell=0}^{\ell_{0}-1}i^{-\ell}j_{\ell}(r)Z_{\theta
    }^{(\ell)}\left( \theta^{\prime
      }\right)\right] \\
  &\notag\qquad+\sum_{\ell=\ell_{0}}^{\infty }
  \Big(i^{-b_{l}}j_{-a_{\ell}}(r)-i^{-(\ell+\frac{1}{2})}j_{\ell}(r)\Big)
  Z_{\theta }^{(\ell)}\left( \theta^{\prime }\right).
\end{align}
The first term at the right hand side of (\ref{suma2}) is clearly bounded,
since it is the difference between a plane wave and the first $(\ell_{0}-1)$
terms of its Jacobi-Anger expansion.

We first notice that the second term at the right hand side of
(\ref{suma2}) is bounded for $r\leq\delta$  if $\delta>0$ is sufficiently small.
Indeed from the estimates
\begin{align}
 \label{eq:stimaJ} &|J_\nu(t)|\leq \frac{1}{\Gamma(1+\nu)}\Big(\frac t2\Big)^\nu
  e^{t^2/4},\quad \text{for all }\nu>0,\ t\geq0,\\
 \notag &|Z_{\theta }^{(\ell)}\left( \theta^{\prime
}\right)|\leq Z_{\theta }^{(\ell)}(\theta)= \frac{2\ell+1}{4\pi
},\quad \text{for all }\ell\geq0,\ \theta,\theta'\in{\mathbb S}^2,
\end{align}
see for example \cite{MU}, it follows that, if $r\leq\delta$,
\begin{align*}
&\bigg|\sum_{\ell=\ell_{0}}^{\infty
}i^{-b_\ell}j_{-a_{\ell}}(r)Z_{\theta }^{(\ell)}(\eta )\bigg|\leq
\sum_{\ell=\ell_{0}}^{\infty }\dfrac{2\ell+1}{4\pi \Gamma ({b_{\ell}+1})}\dfrac{(
\frac{r}{2})^{b_{\ell}}}{r^{\frac{1}{2}}}e^{r^2/4}\\
&\leq
\frac{e^{\delta^2/4}}{\sqrt2\, 4\pi}
\sum_{\ell=\ell_{0}}^{\infty }\dfrac{2\ell+1}{\Gamma
  ({b_{\ell}+1})}\Big(\frac r2\Big)^{-a_\ell}
\leq
\frac{e^{\delta^2/4}}{\sqrt2\, 4\pi}\Big(\frac r2\Big)^{-a_{\ell_0}}
\sum_{\ell=\ell_{0}}^{\infty }\dfrac{2\ell+1}{\Gamma
  (b_{\ell}+1)}\leq C r^{-a_{\ell_0}}
\end{align*}
for some constant $C>0$ dependent on $\delta$ and $\ell_0$ but
independent of $r,\theta,\theta'$.

Next, for $r>\delta $, we write
\begin{align}\label{suma4}
\sum\limits_{\ell=\ell_{0}}^{\infty }&(i^{-b_{\ell}}j_{-a_{\ell}}(r)-i^{-(\ell+
\frac{1}{2})}j_{\ell}(r))Z_{\theta }^{(\ell)}\left( \theta^{\prime }\right) \\
&\notag=\dfrac{1}{2\pi ir^{\frac{1}{2}}}\displaystyle\int_{\gamma }e^{\frac{r}{2}
\left( z-\frac{1}{z}\right) }\left( \sum\limits_{\ell=\ell_{0}}^{\infty } \left[
(iz)^{\ell+\frac{1}{2}-b _{\ell}}-1\right]\frac{Z_{\theta }^{(\ell)}\left(
\theta^{\prime }\right)}{(iz)^{\ell+\frac{1}{2}}} \right) \frac{dz}{z},
\end{align}
where we have used the following representation for Bessel functions%
\begin{equation*}
J_{\nu }(r)=\frac{1}{2\pi i}\int_{\gamma }e^{\frac{r}{2}\left( z-\frac{1}{z}%
\right) }\frac{dz}{z^{\nu +1}},
\end{equation*}%
with $\gamma $ being the positive oriented contour represented in Figure \ref{fig:pspic} (see
\cite[5.10.7]{lebedev}). We have also exchanged sum and integral, which is allowed for any $r, \theta, \theta'$, as we will see below.

For convenience, we split the integral along $\gamma $ into the integrals $%
I_{1}$, along the circumference of radius $1$ (to be denoted as $\Gamma _{1}$%
), and the integral $I_{2}$, along the lines running between $z=-\infty $
and $z=-1$ (to be denoted as $\Gamma _{2}$):%
\begin{equation*}
\int_{\gamma }=\int_{\Gamma _{1}}+\int_{\Gamma _{2}}\equiv I_{1}+I_{2}.
\end{equation*}%
Notice that, by analyticity of the integrand outside $z=\mathbb{R}%
^{-}+i0^{\pm }$, we can write%
\begin{equation*}
\int_{\gamma }=\int_{\Gamma _{1}^{\varepsilon }}+\int_{\Gamma
_{2}^{\varepsilon }}=\lim_{\varepsilon \rightarrow 0^{+}}\left( \int_{\Gamma
_{1}^{\varepsilon }}+\int_{\Gamma _{2}^{\varepsilon }}\right)
\end{equation*}%
where $\Gamma _{1}^{\varepsilon }$ is the circumference of radius $%
1+\varepsilon $ around the origin and $\Gamma _{2}^{\varepsilon }$ runs
along $\left( -\infty ,-1-\varepsilon \right) +i0^{\pm }$. Notice that, for the integral along $\Gamma_{1}^{\varepsilon }\cup \Gamma_{2}^{\varepsilon }$ since $|z|>1$, one has absolute convergence for any given $r, \theta, \theta'$ and hence the exchange of integral and sum performed in formula \eqref{suma4} is allowed by Fubini's Theorem.

We start
estimating the integral along $\Gamma _{1}$. Taking into account that
\begin{equation*}
b_{\ell }-\big(\ell +\tfrac{1}{2}\big)=\sqrt{\left( \ell +\tfrac{1}{2}%
\right) ^{2}+a}-\tfrac{1}{2}-\ell =\tfrac{a}{2\ell +1}+O(\ell ^{-3}),
\end{equation*}%
we have that 
\begin{align}
\left[ (iz)^{(\ell +\frac{1}{2})-b_{\ell }}-1\right] & =-\frac{a}{2\ell +1}%
\log (iz)+\frac{a^{2}}{2}\frac{(\log (iz))^{2}}{(2\ell +1)^{2}}+\frac{O(1)}{%
\ell ^{3}}  \label{suma5} \\
& \equiv J_{1,1}(z,\ell )+J_{1,2}(z,\ell )+J_{1,3}(z,\ell )  \nonumber
\end{align}%
as $\ell \rightarrow +\infty $ uniformly with respect to $z\in \Gamma _{1}$.
Since $z^{-b_{\ell }}$ and $z^{(\ell +\frac{1}{2})}$ have a branch-cut at $z\in
\mathbb{R}^{-}$, the function $\log (iz)$ will also have a branch-cut at $z\in
\mathbb{R}^{-}$, as well as the function $(iz)^{\frac{1}{2}}$ that will
appear below. From (\ref{suma5}) the contribution of the right hand side of (
\ref{suma4}) on $\Gamma _{1}$ can be written
\begin{equation*}
I_{1}=\dfrac{1}{2\pi ir^{\frac{1}{2}}}\displaystyle\int_{\Gamma _{1}}e^{%
\frac{r}{2}\left( z-\frac{1}{z}\right) }\left( \sum\limits_{\ell =\ell
_{0}}^{\infty }\Big(J_{1,1}(z,\ell )+J_{1,2}(z,\ell )+J_{1,3}(z,\ell )\Big)%
\frac{Z_{\theta }^{(\ell )}\left( \theta ^{\prime }\right) }{(iz)^{\ell +%
\frac{1}{2}}}\right) \frac{dz}{z}
\end{equation*}%
\begin{equation*}
\equiv \mathcal{J}_{1,1}+\mathcal{J}_{1,2}+\mathcal{J}_{1,3}\ ,
\end{equation*}%
where every summand $\mathcal{J}_{1,i},\,i=1,2,3,$ corresponds to the integrand with the corresponding $J_{1,i}$.
Since on $\Gamma _{1}$ we have that $|z|=1$ and then $\big|e^{\frac{r}{2}
\left( z-\frac{1}{z}\right) }\big|=1$, from the estimate $\big|Z_{\theta
}^{(\ell )}\big(\theta ^{\prime }\big)\big|\leq \frac{2\ell +1}{4\pi }$ we
deduce that, if $r>\delta $,
\begin{equation*}
\left\vert \mathcal{J}_{1,3}\right\vert \leq \mathrm{const\,}
r^{-1/2}\sum\limits_{\ell =\ell _{0}}^{\infty }\frac{2\ell +1}{\ell ^{3}}
\leq \mathrm{const\,}\delta ^{-1/2},
\end{equation*}%
and hence $\left\vert \mathcal{J}_{1,3}\right\vert $ is
bounded. Concerning
$\mathcal{J}_{1,1}$, we notice that
\begin{align}\label{serie1}
  -a\log (iz)&\sum_{\ell=0}^{\infty }\frac{Z_{\theta }^{(\ell)}\left(
      \theta ^{\prime }\right)
  }{2\ell+1}(iz)^{-\ell-\frac{1}{2}}=-\frac{a\log (iz)}{4\pi }
  \sum_{\ell=0}^{\infty }P_{\ell}(\theta \cdot \theta ^{\prime })(iz)^{-\ell-\frac{1}{2}}\\
  \notag&=-\frac{1}{4\pi }\frac{a(iz)^{-\frac{1}{2}}\log
    (iz)}{\sqrt{1+2\theta \cdot \theta ^{\prime
      }\frac{i}{z}-\frac{1}{z^{2}}}}=-\frac{1}{4\pi }\frac{a(-iz)^{
      \frac{1}{2}}\log (iz)}{\sqrt{z^{2}+2iz(\theta \cdot \theta
      ^{\prime })-1}}
\end{align}
where we have used the well-known identity (see for example \cite{MU})
\begin{equation}\label{Zonal}
  4\pi \frac{Z_{\theta }^{(\ell)}\left( \theta ^{\prime }\right) }{2\ell+1}
  =P_{\ell}(\theta \cdot \theta ^{\prime }),
\end{equation}
with $P_{\ell}$ being the Legendre polynomial of index $\ell$, and the identity (see
Formula 22.9.12 in \cite{AS})
\begin{equation}\label{Legendre}
  \sum_{\ell=0}^{\infty }P_{\ell}(t)w^{\ell}=\frac{1}{\sqrt{1-2wt+w^{2}}},
\end{equation}
which is valid for $\left\vert w\right\vert <1$. Hence, identity (\ref
{serie1}) is valid for $\left\vert z\right\vert >1$. Therefore
\begin{equation}\label{suma3}
  \mathcal{J}_{1,1}\sim (Bounded\,\,\,\,terms)-\frac{a}{8\pi ^{2}ir^{\frac{1}{2
      }}}\int_{\Gamma _{1}^{\varepsilon }}e^{\frac{r}{2}\left( z-\frac{1}{z}
    \right) }\frac{\log (iz)(-iz)^{\frac{1}{2}}}{\sqrt{z^{2}+2iz(\theta \cdot
      \theta ^{\prime })-1}}\frac{dz}{z},
\end{equation}
where the first term at the right-hand side of (\ref{suma3})
represents a finite sum of terms, that are needed to complete the
series (\ref{serie1}) from $\ell=0$ to $\ell=\ell_{0}-1$, and which
are uniformly bounded.  Since $|e^{\frac{r}{2}( z-1/{z}) }|=1$ for all
$r>0$ and $z\in\Gamma _{1}$, if $1-(\theta \cdot \theta ^{\prime })^2$
does not approach zero, i.e. if $\theta \cdot \theta ^{\prime }$ stays
far away from $\pm1$, the second term at the right hand side of
(\ref{suma3}) is uniformly bounded with respect to $\e\to0^+$,
$r>\delta$, and $1-(\theta\cdot\theta')^2>\delta$, due to the
integrability of the two square root singularities of the integrand at
$z_{\pm}=-i(\theta \cdot \theta ^{\prime })\pm \sqrt{
  1-(\theta\cdot\theta^{\prime })^{2}}$.

When $(\theta \cdot \theta ^{\prime })=\mp1$, the two square root
singularities at $z_{\pm }$ collapse into a stronger singularity at
$z=\pm i$. Let us discuss e.g. the case $(\theta \cdot \theta ^{\prime
})=-1$ (the case $(\theta \cdot \theta ^{\prime
})=1$ can be treated similarly); then
\begin{multline}\label{singint}
  \lim_{\varepsilon \rightarrow 0^{+}}\int_{\Gamma _{1}^{\varepsilon
    }}e^{ \frac{r}{2}\left( z-\frac{1}{z}\right)
  }\frac{(-iz)^{\frac{1}{2}}\log (iz)}{
    z-i}\dfrac{dz}{z}\\
  =\pi ^{2}ie^{ir}+PV\int_{\Gamma _{1}}e^{\frac{r}{2}\left( z-
      \frac{1}{z}\right) }\frac{\log
    (iz)}{z-i}(-iz)^{\frac{1}{2}}\dfrac{dz}{z}.
\end{multline}
Equation (\ref{singint}) is simply the Plemelj-Sokhotskyi formula (see
for instance \cite{AF}) for the limit of Cauchy integrals when
approaching a singular point. The first term at the right hand side of
(\ref{singint}) is clearly bounded. The second term at the right hand
side of (\ref{singint}) is a singular integral of the function
$e^{\frac{r}{2}( z-\frac{1}{z})}\frac{\log (iz)(-iz)^{1/2}}{z}$ which
is differentiable for $z=e^{i\theta }$ with $\theta $ in the
neighborhood of $\frac{\pi }{2}$ (remind that the discontinuity of
the argument of $z$ is along the negative real line).

Hence, since the principal value of a Cauchy integral of a
differentiable function is bounded (cf. \cite{AF}), we conclude the
boundedness of $\mathcal{J}_{1,1}$ for any $r>\delta $.
The fact that the principal value integral is bounded for any $r$ does not
exclude the possibility of its diverging as $r\rightarrow \infty $. In order
to exclude this possibility, we consider a neighborhood in $\Gamma _{1}$ of $%
z=i$:
\begin{equation*}
\Gamma _{1}^{s_{0}}=\left\{ z=e^{i\left( \frac{\pi }{2}+s\right) },\
\left\vert s\right\vert <s_{0}\ll 1\right\}
\end{equation*}
and we integrate there for any $r\gg 1$ having into account that
$$
\frac{(\pi +s)e^{i
\frac{s}{2}}}{i(e^{is}-1)}=-\frac{\pi}{s}+ O(1)\quad\text{for } s\thicksim 0.
$$
Hence,
\begin{align*}
  PV\displaystyle\int\limits_{-s_{0}}^{s_{0}}e^{ir\cos s}\frac{(\pi
    +s)e^{i \frac{s}{2}}}{i(e^{is}-1)}ds&= PV\displaystyle\int
  \limits_{-s_{0}}^{s_{0}}e^{ir\cos s} \big(-\frac{\pi}{s}+
  O(1)\big)\,ds\\
  &=-PV\displaystyle\int\limits_{-s_{0}}^{s_{0}}\pi\frac{ e^{ir\cos
      s}}{s}\, ds +PV\displaystyle\int\limits_{-s_{0}}^{s_{0}} O(1)
  e^{ir\cos s}\,ds.
\end{align*}
Since $PV\int\limits_{-s_{0}}^{s_{0}}\pi\frac{ e^{ir\cos
    s}}{s}\, ds=0$, it follows that
\begin{equation*}
  PV\displaystyle\int\limits_{-s_{0}}^{s_{0}}e^{ir\cos s}\frac{(\pi
    +s)e^{i \frac{s}{2}}}{i(e^{is}-1)}ds=O(1).
\end{equation*}
Hence, the integral
along $\Gamma _{1}^{s_{0}}$ is uniformly bounded for any $r.$ The
integral over $\Gamma _{1}\backslash \Gamma _{1}^{s_{0}}$ is also
uniformly bounded since $\big\vert e^{\frac{r}{2}\left(
      z-\frac{1}{z}\right) }\big\vert =1$. Hence, the principal
value integral over $\Gamma _{1}$ is uniformly bounded.

If one considers the two singularities $z_{\pm }$ sufficiently close,
then, similarly to (\ref{singint}), the integral over $\Gamma _{1}$ can
be written as%
\begin{equation}\label{intdelta}
\int_{\Gamma _{1}}=\int_{Arc(z_{-},z_{+})}+\int_{\Gamma _{1}\backslash
Arc(z_{-},z_{+})}  
\end{equation}%
where $Arc(z_{-},z_{+})$ is the small arc of $\Gamma _{1}$ between $z_{+}$
and $z_{-}$. The second integral at the right hand side of (\ref{intdelta})
can be easily estimated just like the principal value above and yields the
same estimates uniformly in $\theta,\theta'$. The first term at the right hand side
of (\ref{intdelta}) is, after writing 
\begin{equation*}
  z_\pm=e^{i\left(\frac\pi2\pm\overline s\right)}
  =-i(\theta\cdot\theta')\pm\sqrt{1-(\theta\cdot\theta')^2},
\end{equation*}

the integral of 
\begin{align*}
  &
  \frac{e^{ir\cos s}(s+\pi)e^{i\frac s2}}
  {\sqrt{\left(e^{i\left(s+\frac\pi2\right)}-
  e^{i\left(\frac\pi2-\overline s\right)}\right)
  \left(e^{i\left(s+\frac\pi2\right)}-
  e^{i\left(\overline s+\frac\pi2\right)}\right)}}
  \\
  &
  =
  \frac{\pi e^{ir\cos s}}{\sqrt{|s+\overline s||s-\overline s|}}(1+O(s-\overline s)+O(s+\overline s))
  \times
  \begin{cases}
    -1,
    \quad\text{if }
    s<-\overline s
    \\
    +1
    \quad\text{if }
    s>\overline s
    \\
    -i
    \quad\text{if }
    -\overline s<s<\overline s
  \end{cases}
  \\
  &
  =
  \nu(s)
  \frac{\pi e^{ir\cos s}}{\sqrt{|s+\overline s||s-\overline s|}}+O\left(\frac1{\sqrt{|s-\overline s|}}\right)
  +O\left(\frac1{\sqrt{|s+\overline s|}}\right),
\end{align*}
where $|\nu(s)|=1$. As a consequence, 
we can estimate
\begin{align*}
  \left|\int_{Arc(z_-,z_+)}\right|
  &
  =
  \left|\int_{-\overline s}^{\overline s}
  \frac{\pi e^{ir\cos s}}{\sqrt{\overline s^2-s^2}}ds+O(1)\right|
  \\
  &
  =
  \left|\int_{-1}^{1}
  \frac{\pi e^{ir\cos(\overline st)}}{\sqrt{1-t^2}}dt+O(1)\right|
  \\
  &
  \leq {\rm const},
\end{align*}
uniformly with respect to $r$ and $\overline s$.
Therefore, we conclude that the integral on $\Gamma_1$ is
uniformly bounded both in $\delta $ and $r$.

Finally, the term $%
J_{1,2}$ in (\ref{suma5}), inserted at the right hand side of (\ref{suma4}),
produces
$$
\dfrac{1}{2\pi ir^{\frac{1}{2}}}\displaystyle\int_{\Gamma _{1}}e^{%
\frac{r}{2}\left( z-\frac{1}{z}\right) }\frac{a^{2}(\log (iz))^{2}}{2}\left( \sum\limits_{\ell =\ell
_{0}}^{\infty }
\frac{Z_{\theta }^{(\ell )}\left( \theta ^{\prime }\right)}{(2\ell +1)^{2}} {(iz)^{-\ell -%
\frac{1}{2}}}\right) \frac{dz}{z}
$$
where the series%
\begin{equation*}
F(\theta,\theta ^{\prime },z)=\sum_{\ell =0}^{\infty }\frac{Z_{\theta }^{(\ell
)}\left( \theta ^{\prime }\right) }{(2\ell +1)^{2}}(iz)^{-\ell -\frac{1}{2}}=\dint g(\theta,\theta', z) \,dz
\end{equation*}%
is the primitive in $z$ of the series%
\begin{equation*}
g(\theta,\theta',z)=-\frac{i}{2}\sum_{\ell =0}^{\infty }\frac{Z_{\theta }^{(\ell )}\left( \theta
^{\prime }\right) }{2\ell +1}(iz)^{-\ell -\frac{3}{2}}.
\end{equation*}%
Thus, using \eqref{Zonal} and \eqref{Legendre}, we conclude that $F(\theta,\theta ^{\prime },z)$ is the primitive of
$$
g(\theta,\theta',z)=-\frac{i}{8\pi }\frac{%
(iz)^{-\frac{3}{2}}}{\sqrt{1+2(\theta \cdot \theta ^{\prime })\frac{i}{z}-%
\frac{1}{z^{2}}}}
$$
and hence, since $g(\theta,\theta',z)$ presents a square root singularity if $\theta \cdot
\theta ^{\prime }\neq -1$ or a $1/(z-i)$ singularity at $z=i$ if $\theta
\cdot \theta ^{\prime }=-1$, we conclude that $F(\theta,\theta',z)$ presents at most a
log-type singularity, which is integrable, and consequently the integral yields a uniformly bounded contribution $\mathcal{J}_{1,2}$.
Therefore, we conclude that $I_{1}$ is uniformly bounded.

We continue estimating $I_{2}$,
$$
I_{2}=\dfrac{1}{2\pi ir^{\frac{1}{2}}}\displaystyle\int_{\Gamma_2 }e^{\frac{r}{2}
\left( z-\frac{1}{z}\right) }\left( \sum\limits_{\ell=\ell_{0}}^{\infty } \left[
(iz)^{\ell+\frac{1}{2}-b _{\ell}}-1\right]\frac{Z_{\theta }^{(\ell)}\left(
\theta^{\prime }\right)}{(iz)^{\ell+\frac{1}{2}}} \right) \frac{dz}{z}.
$$
Introducing the changes of variables, $z=e^{\pm \pi i}e^{t}$, exchanging sum and integral  (arguing as above)
and rearranging terms, we rewrite it in the form%
\begin{equation*}
I_{2}=\dfrac{1}{2\pi ir^{\frac{1}{2}}}\sum_{\ell =\ell _{0}}^{\infty }Z_{\theta }^{(\ell )}\left( \theta
^{\prime }\right) (A_{\ell }(r)+B_{\ell }(r))\equiv \dfrac{1}{2\pi ir^{\frac{1}{2}}}\big(\mathcal{J}_{2,1}+%
\mathcal{J}_{2,2})
\end{equation*}%
where
\begin{equation*}
A_{\ell }(r)=-\dfrac{2\sin {(\pi b_{\ell })}}{i^{b_{\ell }-1}}\displaystyle%
\int_{0}^{\infty }e^{-r\sinh t}(e^{-b_{\ell }t}-e^{-(\ell +\frac{1}{2})t})dt
\end{equation*}%
and
\begin{equation*}
B_{\ell }(r)=\displaystyle -2i\int_{0}^{\infty }e^{-r\sinh t}e^{-(\ell +\frac{1%
}{2})t}\bigg(\dfrac{\sin {\pi b_{\ell }}}{i^{b_{\ell }}}-\dfrac{\sin {\pi
(\ell +\frac{1}{2})}}{i^{\ell +\frac{1}{2}}}\bigg)\,dt.
\end{equation*}%
We estimate $A_{\ell }$ by using again that $|Z_{\theta }^{(\ell )}\left(
\theta ^{\prime }\right) |\leq \frac{2\ell +1}{4\pi }$,
\begin{align*}
\left\vert \mathcal{J}_{2,1}\right\vert &=\left\vert \sum_{\ell =\ell
_{0}}^{\infty }Z_{\theta }^{(\ell )}\left( \theta ^{\prime }\right) A_{\ell
}(r)\right\vert \leq C\sum_{\ell =\ell _{0}}^{\infty }(2\ell
+1)\int_{0}^{\infty }e^{-r\sinh t}\left\vert e^{-b_{\ell }t}-e^{-(\ell +
\frac{1}{2})t}\right\vert dt\\
&\leq C\sum_{\ell_{0}}^{\infty }(2\ell+1)\left\vert \frac{1}{b _{\ell}}-\frac{1}{\ell+
\frac{1}{2}}\right\vert \leq C.
\end{align*}
In order to estimate $B_{\ell }$, notice that
\begin{equation*}
\bigg(\dfrac{\sin {\pi b_{\ell }}}{i^{b_{\ell }}}-\dfrac{\sin {\pi (\ell +%
\frac{1}{2})}}{i^{\ell +\frac{1}{2}}}\bigg)
\end{equation*}%
\begin{equation*}
=\frac{1}{i^{{\ell +\frac{1}{2}}}}\left[ \sin {\pi b_{\ell }}-\sin {\pi
(\ell +\frac{1}{2})}\right] -\sin {\pi b_{\ell }}\left[ \frac{1}{i^{{\ell +%
\frac{1}{2}}}}-\frac{1}{i^{b_{\ell }}}\right]
\end{equation*}%
\begin{equation*}
=-\frac{(-1)^{{\ell }}}{i^{{\ell +\frac{1}{2}}}}\frac{1}{2}\left( \frac{a{%
\pi }}{2{\ell +1}}\right) ^{2}-\frac{(-1)^{{\ell }}}{i^{{\ell +\frac{1}{2}}}}%
\left( \frac{a{\pi }}{2(2\ell +1)}i+\frac{a^{2}{\pi }^{2}}{8(2{\ell +1)}^{2}}%
\right)+O(\ell
^{-3})
\end{equation*}%
\begin{equation*}
=-\frac{i^{{\ell +\frac{1}{2}}}}{2{\ell +1}}\frac{a{\pi }}{2}-\frac{5 i^{{\ell -\frac{1}{2}}}}{8} \bigg(\frac{a\pi}{(2{\ell +1)}}\bigg)^2+O(\ell
^{-3}).
\end{equation*}%
By using formula \eqref{Legendre}, it readily follows
\begin{align*}
  \left\vert \mathcal{J}_{2,2}\right\vert =&\left\vert \sum_{\ell
      =\ell _{0}}^{\infty }Z_{\theta }^{(\ell )}\left( \theta ^{\prime
      }\right) B_{\ell }(r)\right\vert \leq \left\vert
    \displaystyle\int_{0}^{\infty }e^{-r\sinh
      t}e^{-\frac{t}{2}}\times \right.\\
  & \times \Bigg[K_{\ell_{0}}(r,\theta,\theta',t)- i^{\frac{1}{2}}\dfrac{a\pi }{4}\dfrac{1}{\sqrt{1-ie^{-t}(\theta \cdot \theta ^{\prime })-e^{-2t}}}\\
  &\quad -\sum_{\ell =\ell _{0}}^{\infty }\frac{5a^{2}\pi ^{2}i^{{\ell
        -\frac{1}{2}}}e^{-{\ell t} }Z_{\theta }^{(\ell
      )}(\theta')}{16}\left( \frac{1}{(2{\ell +1)} ^{2}}+O(\ell
    ^{-3})\right) \Bigg]\Bigg\vert ,
\end{align*}
where $K_{\ell_{0}}$ accounts for the terms that need to be added in
order to use \eqref{Legendre} and which is uniformly bounded.

Since $\sqrt{|1-ie^{-t}(\theta \cdot \theta ^{\prime })-e^{-2t}|}\geq
\frac{{\rm const\,}\sqrt{t}}{1+\sqrt{t}}$ for some ${\rm const\,}>0$ and, using again that $|Z_{\theta
}^{(\ell )}\left( \theta ^{\prime }\right) |\leq \frac{2\ell +1}{4\pi
}$,
\begin{equation*}
\left\vert \sum_{\ell =\ell _{0}}^{\infty }i^{{\ell }}e^{-{\ell t}%
}Z_{\theta }^{(\ell )}\left(
\theta ^{\prime }\right)\left( \frac{1}{(2{\ell +1)}%
^{2}}+O(\ell ^{-3})\right) \right\vert \leq C\sum_{\ell =\ell _{0}}^{\infty }%
\frac{e^{-{\ell t}}}{{\ell }}\leq 2C\log \left\vert t\right\vert
\end{equation*}%
for some constat $C$ and any $\theta, \theta ^{\prime }$, we conclude
the existence of other constants $C^{\prime },C''$ such that
\begin{equation*}
\left\vert \mathcal{J}_{2,2}\right\vert \leq \frac{C^{\prime }}{4}
\int_{0}^{\infty }e^{-r\sinh t}e^{-\frac{t}{2}}\left[ \dfrac{1+\sqrt{t}}{
\sqrt{t}}+\log \left\vert t\right\vert \right] \leq C''.
\end{equation*}
Hence the uniform boundedness of $\mathcal{J}_{2,2}$ follows. We conclude then
\begin{equation}\label{eq:supS2}
\sup_{\substack{ r\geq 0 \\ \theta ,\theta ^{\prime }\in {\mathbb{S}}^{N-1}}}%
|S_{2}(r,\theta ,\theta ^{\prime })|<+\infty .
\end{equation}
If $\alpha_1=a_0<0$, then $\ell_0=0$. Hence $S_1(r,\theta
,\theta')=0$ and \eqref{eq:estimate_i_K} is proved.

\noindent If $\alpha_1=a_0>0$, then $\ell_0>0$. From \eqref{eq:stimaJ} and the
fact that $a_\ell\leq a_{0}$ for all $\ell\in\N$, we
deduce that
\begin{equation}\label{eq:supS11}
|S_1(r,\theta
,\theta')|\leq {\rm const\,}r^{-a_0}={\rm
  const\,}r^{-\alpha_1}\quad\text{for all }r\leq1,\
\theta,\theta'\in\SN.
\end{equation}
On the other hand, from \eqref{eq:asintinfti} and \eqref{ineqs} we
easily deduce that
\begin{equation}\label{eq:supS12}
|S_1(r,\theta
,\theta')|\leq {\rm const\,}\quad\text{for all }r\geq1,\
\theta,\theta'\in\SN.
\end{equation}
Estimate \eqref{eq:estimate_ii_K} follows from \eqref{suma},
\eqref{eq:supS2}, \eqref{eq:supS11}, and \eqref{eq:supS12}.\qed

\begin{figure}[h]
\begin{pspicture}(-4,-2.5)(4,2.5)
\psset{arrowscale=1.7}
\psline[linewidth=0.02cm](-4,0)(4,0)
\psline[linewidth=0.02cm](0,-2.5)(0,2.5)
\psline[linewidth=0.04cm]{->}(-4,-0.2)(-0.95,-0.2)
\psline[linewidth=0.04cm]{->}(-4,-0.2)(-2,-0.2)
\psline[linewidth=0.04cm]{->}(-4,-0.2)(-3,-0.2)
\psline[linewidth=0.04cm]{<-}(-4,0.2)(-1,0.2)
\psline[linewidth=0.04cm]{<-}(-3,0.2)(-1,0.2)
\psline[linewidth=0.04cm]{<-}(-2,0.2)(-1,0.2)
 \usefont{T1}{ptm}{m}{n}
 \rput(1,1){{$\bf\Gamma_1$}}
 \usefont{T1}{ptm}{m}{n}
 \rput(-2.5,0.5){{$\bf\Gamma_2$}}
\psarc[linewidth=0.04]{->}(0,0){1.001998}{190.6}{169.6900423}
\psarc[linewidth=0.04]{->}(0,0){1.001998}{190.6}{90}
\psarc[linewidth=0.04]{->}(0,0){1.001998}{190.6}{0}
\psarc[linewidth=0.04]{->}(0,0){1.001998}{190.6}{-90}
\end{pspicture}
\caption{Integration oriented domain $\gamma$.}
\label{fig:pspic}
\end{figure}
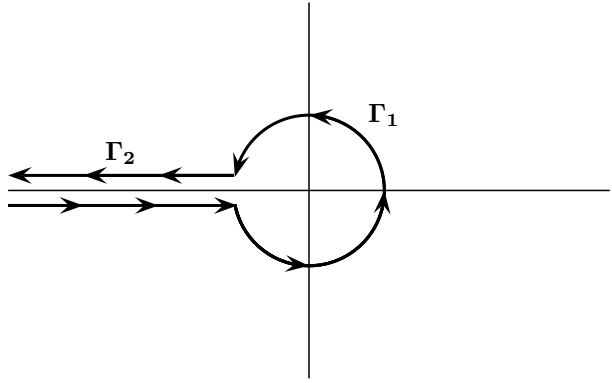

\appendix
 \section*{Appendix}
 \setcounter{section}{1}
 \setcounter{Theorem}{0}

The asymptotic behavior of eigenvalues $\mu_k(\A,a)$ as $k\to+\infty$ is
described by Weyl's law, which is recalled in the theorem below. We
refer to \cite{reedsimon4, SafarovVassiliev} for a proof.
\begin{Theorem}[\bf Weyl's law]\label{t:weyl}
For $a\in L^{\infty }({\mathbb{S}}^{N-1},{\mathbb{R}})$ and
${\mathbf{A}}\in C^{1}({\mathbb{S}}^{N-1},{\mathbb{R}}^{N})$, let $\{\mu_k(\A,a)\}_{k\geq1}$ be the eigenvalues of
the operator $L_{\A,a}=\big(-i\,\nabla_{\mathbb{S}^{N-1}}+{\mathbf{A}}\big)%
^2+a(\theta)$. Then
\begin{equation}\label{eq:weyl}
\mu_k(\A,a)=C(N,\A,a)k^{2/(N-1)}\big(1+o(1)\big)\quad\text{as }k\to+\infty,
\end{equation}
for some positive constant $C(N,\A,a)$ depending only on $N$, $\A$, and $a$.
\end{Theorem}

The following lemma provides an estimate of the $L^{\infty}$-norm of
eigenfunctions of the operator $L_{\A,a}$ in terms of the corresponding eigenvalues.
\begin{Lemma}\label{l:stiautf}
For $a\in L^{\infty }({\mathbb{S}}^{N-1},{\mathbb{R}})$,
${\mathbf{A}}\in C^{1}({\mathbb{S}}^{N-1},{\mathbb{R}}^{N})$, and $k\in\N\setminus\{0\}$,
let $\psi_k$ be a $L^2$-normalized
eigenfunction of the Schr\"odinger operator $L_{\A,a}$ on the sphere associated to
the $k$-th eigenvalue $\mu_k(\A,a)$, i.e. satisfying \eqref{angular}.
Then, there exists a constant $\widetilde C$ depending only on $N$, $\A$,
and $a$ such that
$$
|\psi_k(\theta)|\leq \widetilde C\, |\mu_k|^{\lfloor (N-1)/2\rfloor},
$$
where $\lfloor \cdot\rfloor$ denotes the floor function, i.e. $\lfloor
x\rfloor:=\max\{j\in\Z:\ j\leq x\}$.
\end{Lemma}

\begin{pf}
Using classical elliptic regularity theory and bootstrap methods, we
can easily prove that for any $j\in\N$ there exists a constant
$C(N,\A,a,j)$, depending only on $j$, $\A$, $a$, and $N$ but
independent of $k$, such that,
for large $k$,
$$
\|\psi_k\|_{W^{2,\frac{2(N-1)}{(N-1)-2(j-1)}}({\mathbb S}^{N-1})}\leq
C(N,\A,a,j) \big(\mu_k(\A,a)\big)^j.
$$
Choosing $j=\big\lfloor \frac{N-1}2\big\rfloor$, by Sobolev's inclusions we
deduce that
$$
W^{2,\frac{2(N-1)}{(N-1)-2(j-1)}}({\mathbb
S}^{N-1})\hookrightarrow C^{0,\alpha}({\mathbb
S}^{N-1})\hookrightarrow L^{\infty}({\mathbb S}^{N-1}),
$$
for any $0<\alpha<1-\frac {N-1}2+\big\lfloor\frac
{N-1}2\big\rfloor$, thus implying the required estimate.
\end{pf}
\section*{Acknowledgments}

The authors would like to thank J.J. Vel\'azquez and F. Maci\`a for
fruitful conversations on the topic of this article.
 This work has been partially supported by Grant MTM2011-26016 and
 MTM2010-18128, and by the PRIN2009 Grant ''Critical Point Theory and 
Perturbative Methods for Nonlinear Differential Equations''.


\begin{thebibliography}{99}

\bibitem{AF} {\sc Ablowitz, M.J. and Fokas, A.S.},
  \textit{Introduction and Applications of Complex Variables},
  Cambridge University Press, second edition (2003).

\bibitem{AS}
{\sc Abramowitz, M. and Stegun, I. A.},
    \textit{Handbook of mathematical functions with formulas, graphs, and
    mathematical tables.} National Bureau of Standards Applied
  Mathematics Series {\bf 55}. For sale by the Superintendent of Documents,
  U.S. Government Printing Office, Washington, D.C. 1964.

\bibitem{B}
    {\sc Beals, M.},
    Optimal $L\sp \infty$ decay for solutions to the wave equation with
    a potential,
    \textit{Comm. Partial Differential Equations} \textbf{19} (1994) no. 7-8, 1319--1369.

\bibitem{BS}
    {\sc Beals, M. and Strauss, W.},
    $L\sp p$ estimates for the wave equation with a potential,
    \textit{Comm. Partial Differential Equations} \textbf{18} (1993) no. 7-8, 1365--1397.

\bibitem{BG} {\sc Beceanu, M. and Goldberg, M.}, Decay estimates for
    the Sch\"odinger equation with critical potentials, to appear in
    \textit{Comm. Math. Phys.}, arXiv:1009.5285.

\bibitem{BeStr}
{\sc Bezubik, A. and Strasburger, A.},
A new form of the spherical expansion of zonal functions and Fourier transforms of
$\mathop{rm SO}(d)$-finite functions, \textit{SIGMA Symmetry Integrability Geom. Methods Appl.}
\textbf{2} (2006), Paper 033, 8 pp.

\bibitem{BPSTZ1}
{\sc Burq, N., Planchon, F., Stalker, J., and Tahvildar-Zadeh, S.},
Strichartz estimates for the wave and Schr\"odinger equations with
the inverse-square potential, \textit{J. Funct. Anal.} \textbf{203}
(2003) no. 2, 519--549.

\bibitem{BPSTZ}
{\sc Burq, N., Planchon, F., Stalker, J., and Tahvildar-Zadeh, S.}
Strichartz estimates for the wave and {S}chr\"odinger equations with
potentials of critical decay, \textit{Indiana Univ. Math. J.}
\textbf{53} (2004) no. 6, 1665--1680.

\bibitem{C}
    {\sc Cuccagna, S.},
    On the wave equation with a potential,
    \textit{Comm. Partial Differential Equations} \textbf{25} (2000) no. 7--8, 1549-1565.

\bibitem{CS}
    {\sc Cuccagna, S. and Schirmer, P.},
    On the wave equation with a magnetic potential,
    \textit{Comm. Pure Appl. Math.} \textbf{54} (2001) no. 2, 135--152.

  \bibitem{DF1} {\sc D'Ancona, P. and Fanelli, L.}, $L^p$-boundedness
    of the wave operator for the one dimensional Schr\"odinger
    operators, \textit{Comm. Math. Phys.} {\bf 268} (2006), 415--438.

\bibitem{DF2}
    {\sc D'Ancona, P. and Fanelli, L.},
    Decay estimates for the wave and Dirac equations with a magnetic potential,
    \textit{Comm. Pure Appl. Math.} {\bf 60} (2007), 357--392.

\bibitem{DF3}
{\sc D'Ancona, P. and Fanelli, L.},
Strichartz and smoothing estimates for dispersive equations with magnetic
potentials, \textit{Comm. Part. Diff. Eqns.} {\bf 33} (2008), 1082--1112.

\bibitem{DFVV}
{\sc D'Ancona, P., Fanelli, L., Vega, L., and Visciglia, N.}, Endpoint Strichartz estimates for the magnetic
Schr\"odinger equation, \textit{J. Funct. Anal.} {\bf 258} (2010), 3227--3240.

\bibitem{DP}
    {\sc D'Ancona, P. and Pierfelice, V.},
    On the wave
    equation with a large rough potential, \textit{J. Func. Analysis} {\bf 227} (2005), 30--77.

\bibitem{EGS1}
{\sc Erdogan, M.B., Goldberg, M., and Schlag, W.},
Strichartz and Smoothing Estimates for Schr\"odinger
Operators with Almost Critical Magnetic Potentials in Three and Higher
Dimensions,
\textit{Forum Math.} {\bf 21} (2009), 687--722.

\bibitem{EGS2}
{\sc Erdogan, M.B., Goldberg, M., and Schlag, W.},
Strichartz and smoothing estimates for Schrodinger operators with
large magnetic potentials in $\R^3$, \textit{J. European Math.
Soc.} {\bf 10} (2008), 507--531.

\bibitem{FG}
{\sc Fanelli, L., and Garc\'ia, A.}, Counterexamples to Strichartz estimates for the magnetic
Schr\"odinger equation, \textit{Comm. Cont. Math.} {\bf 13} (2011)
no. 2, 213--234.

\bibitem{FFT}
{\sc Felli, V. , Ferrero, A. and Terracini, S.}, Asymptotic
    behavior of solutions to Schr\"odinger equations near an isolated
    singularity of the electromagnetic potential,
  \textit{J. Eur. Math. Soc.} {\bf 13} (2011) no. 1, 119--174.

\bibitem{GP}
{\sc Garc\'ia Azorero, J.  and Peral, I.}, Hardy inequalities and
some critical elliptic and parabolic problems,  \textit{J.
Differential Equations} {\bf 144} (1998), 441--476.

\bibitem{GST}
{\sc Georgiev, V., Stefanov, A., and Tarulli, M.} Smoothing -
Strichartz estimates for the Schr\"odinger equation with small
magnetic potential, \textit{Discrete Contin. Dyn. Syst.} A {\bf 17}
(2007), 771--786.

\bibitem{GeVi}
    {\sc Georgiev, V. and Visciglia, N.},
    Decay estimates for the wave equation with potential,
    \textit{Comm. Partial Differential Equations} \textbf{28} (2003) no. 7-8, 1325--1369.

\bibitem{GV1}
{\sc Ginibre, J. and Velo, G.}, Scattering theory in the energy space for a class of nonlinear
Schr\"odinger equations, \textit{J. Math. Pures Appl.} {\bf 64} (1984), 363--401.

\bibitem{GV2}
{\sc Ginibre, J. and Velo, G.}, Generalized Strichartz
inequalities for the wave equation, \textit{J. Funct. Anal.} {\bf 133} no. 1
(1995), 50--68.

\bibitem{GS}
    {\sc Goldberg, M. and Schlag, W.},
    Dispersive estimates for Schr\"odinger operators in dimensions one
    and three,
    \textit{Comm. Math. Phys.} \textbf{251} (2004) no. 1, 157--178.

\bibitem{GVV}
{\sc Goldberg, M., Vega, L., and Visciglia, N.}, Counterexamples of
Strichartz inequalities for Schr\"odinger equations with repulsive
potentials, \textit{Int. Math Res Not.}, 2006 Vol. 2006: article ID
13927.

\bibitem{HLP} 
{\sc Hardy, G., Littlewood, J.E., and Polya, G.} \textit{Inequalities.}
  Reprint of the 1952 edition. Cambridge Mathematical
  Library. Cambridge University Press, Cambridge, 1988.

\bibitem{Ismail}
{\sc Ismail, M. E. H.}, \textit{Classical and quantum orthogonal polynomials in one variable.}
Encyclopedia of Mathematics and its Applications, 98. Cambridge University Press, Cambridge, 2005.

\bibitem{JN}
    {\sc Jensen, A. and Nakamura, S.}, Mapping properties of
    functions of Schr\"odinger operators between $L^p$-spaces and
    Besov spaces, \textit{Adv. Stud. in Pure Math.} {\bf 23} (1994),
    187--209.

\bibitem{JSS}
    {\sc Journ{\'e}, J.-L.., Soffer, A., and Sogge, C.-D.},
    Decay estimates for Schr\"odinger operators,
    \textit{Comm. Pure Appl. Math.} \textbf{44} (1991) no. 5, 573--604.

\bibitem{K}
    {\sc Kato, T.},
    Wave operators and similarity for some
    non-selfadjoint operators,
    \textit{Math. Annalen} \textbf{162} (1966), 258--279.

\bibitem{KW}
{\sc Kavian, O. and Weissler, F. B.}, Self-similar solutions of
    the pseudo-conformally invariant nonlinear Schr\"odinger
    equation, \textit{Michigan Math. J.} {\bf 41} no. 1 (1994), 151--173.

\bibitem{KT}
{\sc Keel, M. and Tao, T.}, Endpoint Strichartz estimates, \textit{Am. J.
Math.} {\bf 120} no. 5 (1998), 955--980.

\bibitem{lw}
{\sc Laptev, A.  and Weidl, T.}, Hardy inequalities for
      magnetic Dirichlet forms, \textit{Mathematical results in quantum
    mechanics (Prague, 1998)}, 299--305; \textit{Oper. Theory Adv. Appl.} {\bf 108},
    Birkh\"auser, Basel, 1999.

\bibitem{lebedev} {\sc Lebedev, N. N.}, \textit{Special functions
      and their applications.}  Revised edition, translated from the
    Russian and edited by Richard A. Silverman. Unabridged and
    corrected republication. Dover Publications, Inc., New York, 1972.

\bibitem{LL}
{\sc Lieb, E. H. and Loss, M.}, \textit{Analysis,} Graduate Studies in
  Mathematics 14, AMS (1997).

\bibitem{MA}
{\sc MacDonald, A. D.}, Properties of the confluent
    hypergeometric function,  \textit{J. Math. Physics} {\bf 28} (1949), 183--191.

  \bibitem{MMT} {\sc Marzuola, J., Metcalfe, J., and Tataru, D.},
    Strichartz estimates and local smoothing estimates for
    asymptotically flat Schr\"odinger equations,
    \textit{J. Funct. Anal.} {\bf 255} (2008), 1497--1553.

  \bibitem{MU} {\sc M\"uller, C.}, Spherical
    harmonics. \textit{Lecture Notes in Mathematics}, {\bf 17}
    Springer-Verlag, Berlin-New York 1966.

\bibitem{PSTZ}
    {\sc Planchon, F. , Stalker, J., and Tahvildar-Zadeh, S.},
    Dispersive
    estimates for the wave equation with the inverse-square potential,
    \textit{Discrete Contin. Dyn. Syst.} \textbf{9} (2003), 1387--1400.

\bibitem{reedsimon4} {\sc Reed, M. and Simon, B.}, {\it Methods of modern
    mathematical physics. IV. Analysis of operators,} Academic Press,
  New York-London, 1978.

\bibitem{RZ}
    {\sc Robbiano, L. and Zuily, C.},
    Strichartz estimates for Schr\"odinger equations with variable
    coefficients, \textit{M\'em. Soc. Math. Fr. (N.S.)} No. 101-102 (2005). 

\bibitem{RS}
    {\sc Rodnianski, I. and Schlag, W.},
    Time decay for solutions of Schr\"odinger equations with rough and
    time-dependent potentials,
    \textit{Invent. Math.} \textbf{155} (2004) no. 3, 451--513.

\bibitem{SafarovVassiliev} {\sc Safarov, Yu. and Vassiliev, D.}, {\it The asymptotic
    distribution of eigenvalues of partial differential operators,}
    Translated from the Russian manuscript by the authors,
    Translations of Mathematical Monographs, 155. American
    Mathematical Society, Providence, RI, 1997.

\bibitem{S}
{\sc Schlag, W.}, Dispersive estimates for Schr\"odinger operators: a survey, \textit{Mathematical aspects
of nonlinear dispersive equations, 255285, Ann. of Math. Stud.}, {\bf 163}, Princeton
Univ. Press, Princeton, NJ, 2007.

\bibitem{Se}
{\sc Segal, I.}, Space-time decay for solutions of wave equations, \textit{Adv. Math.} {\bf 22} no. 3
(1976), 305--311.

\bibitem{ST}
    {\sc Staffilani, G. and Tataru, D.},
    Strichartz estimates for a Schr\"odinger operator with nonsmooth
    coefficients,
    \textit{Comm. Partial Differential Equations} \textbf{27} (2002) no. 7-8,
    1337--1372.

\bibitem{Ste}
{\sc Stefanov, A.}, Strichartz estimates for the magnetic
Schr\"odinger equation, \textit{Adv. Math.} {\bf 210} (2007), 246--303.

\bibitem{St}
{\sc Strichartz, R.}, Restrictions of Fourier transforms to quadratic surfaces and decay of
solutions of wave equation, \textit{Duke Math. J.} {\bf 44} (1977), 705--714.

\bibitem{T}
{\sc Tomas, P.}, A restriction theorem for the Fourier transform, \textit{Bull. Amer. Math. Soc.}
{\bf 81} (1975), 477--478.

\bibitem{watson}
{\sc Watson, G. N.}, {\it A treatise on the theory of Bessel
functions}, 2d ed., Cambridge Univ. Press, London,
England, 1944.

\bibitem{W1}
    {\sc Weder, R.},
    The $W_{k,p}$-continuity of the Schr\"odinger Wave
    Operators on the line, \textit{Comm. Math. Phys.} {\bf 208} (1999),
    507--520.

\bibitem{W2}
    {\sc Weder, R.},
    $L^p-L^{p'}$ estimates for the Schr\"odinger equations on
    the line and inverse scattering for the nonlinear Schr\"odinger
    equation with a potential, \textit{J. Funct. Anal.} {\bf 170} (2000),
    37--68.

\bibitem{Y1}
{\sc Yajima, K.}, Existence of solutions for Schr\"odinger evolution equations, \textit{Comm.
Math. Phys.} {\bf 110} (1987), 415--426.

\bibitem{Y2}
    {\sc Yajima, K.},
    The $W^{k,p}$-continuity of wave operators for Schr\"odinger
    operators,
    \textit{J. Math. Soc. Japan} \textbf{47} (1995) no. 3, 551--581.

\bibitem{Y3}
    {\sc Yajima, K.},
    The $W^{k,p}$-continuity of wave operators for Schr\"odinger
    operators III, even dimensional cases $m\geq4$,
    \textit{J. Math. Sci. Univ. Tokyo} \textbf{2} (1995), 311--346.

\bibitem{Y4}
    {\sc Yajima, K.},
    $L\sp p$-boundedness of wave operators for two-dimensional
    Schr\"odinger operators,
    \textit{Comm. Math. Phys.} \textbf{208} (1999) no. 1, 125--152.

\end{thebibliography}
\end{document}